\documentclass[a4paper,11pt,oneside]{article}
\usepackage[english]{babel}
\usepackage{graphicx}
\usepackage[latin1]{inputenc}
\usepackage{amssymb,amsmath,textcomp,amsthm}
\usepackage{amsfonts}
\usepackage{setspace}
\usepackage{multicol}
\usepackage{multirow}
\usepackage{color}

\usepackage{algorithm}
\usepackage{algo_olm}

\newtheorem{definition}{Definition}[subsection]
\newtheorem{proposition}{Proposition}[subsection]
\newtheorem{property}{Property}[subsection]

\newtheorem{remark}{Remark}[subsection]

\newcommand{\indices}[4]{
\scriptsize
\begin{array}{c}
1\leq #1\leq #2\\
1\leq #3\leq #4
\end{array}
}

\begin{document}
\title{Accurate and efficient evaluation of the a posteriori error estimator in the reduced basis method}%\thanks{...}\thanks{...}% At most 5 thanks
\maketitle

\begin{center}
{Fabien Casenave$^{1}$, Alexandre Ern$^{1}$ and Tony Leli\`{e}vre$^{1,2}$}\\
\end{center}
\begin{center}
\end{center}
\begin{center}
$^1$ Universit\'{e} Paris-Est, CERMICS, \'{E}cole des Ponts-Paristech, 6 \& 8 av Blaise Pascal, 77455 Marne-la-Vall\'{e}e Cedex 2, France
\end{center}
\begin{center}
$^2$ INRIA Rocquencourt, MICMAC Team-Project, Domaine de Voluceau, B.P. 105, 78153 Le Chesnay Cedex, France 
\end{center}

\begin{abstract}  
The reduced basis method is a model reduction technique yielding substantial savings of computational time when a solution to a
parametrized equation has to be computed for many values of the parameter.
Certification of the approximation is possible by means of an a posteriori error bound.
Under appropriate assumptions, this error bound is computed with an algorithm of
complexity independent of the size of the full problem.
In practice, the evaluation of the error bound can become very sensitive to round-off errors. We propose herein an explanation
of this fact. A first remedy has been proposed in [F. Casenave, Accurate \textit{a posteriori} error evaluation in the reduced basis method. \textit{C. R. Math. Acad. Sci. Paris} \textbf{350} (2012) 539--542.].
Herein, we improve this remedy by proposing a new approximation of the error bound using the
Empirical Interpolation Method (EIM). This method achieves higher levels of accuracy and requires potentially less
precomputations than the usual formula.
A version of the EIM stabilized with respect to round-off errors is also derived.
The method is illustrated on a simple one-dimensional diffusion problem and a
three-dimensional acoustic scattering problem solved by a boundary element method.
\end{abstract}

% 65N15 : Numerical analysis : Error bounds
% 65D05 : Numerical analysis : Interpolation
% 68W25 : Computer science For papers involving machine computations and programs in a specific mathematical area : Approximation algorithms
% 76Q05 : Fluid mechanics : Hydro- and aero-acoustics


%%-----------------------------
%%      your text
%%-----------------------------
\section*{Introduction}
In many problems, such as optimization, uncertainty propagation or real-time simulation, one has to evaluate an objective
function for a large number of values of some parameters. Evaluating this objective function often implies solving a
parametrized partial differential equation for a given parameter value.
In an industrial context, one evaluation of the objective function can already be a challenging numerical problem.
To keep reasonable computational costs, various model reduction techniques have been developed to speed up
computations.
We focus on the Reduced Basis (RB) method \cite{1, RB2}. This method has been applied to many kinds of
problems, including nonlinear problems
such as the viscous Burgers equation \cite{2} or the steady incompressible Navier-Stokes equations \cite{3}.

As described in Section \ref{sec:RB}, the RB method consists in replacing the sequence
$\mathcal{P}\ni\mu\overset{E_\mu}{\mapsto}u_\mu\mapsto{Q(u_\mu)}$ by the sequence
$\mathcal{P}\ni\mu\overset{\hat{E}_\mu}{\mapsto}\hat{u}_\mu\mapsto{\hat{Q}(\hat{u}_\mu)}$.
Here, $\mathcal{P}$ denotes the parameter set, ${E_\mu}:\mu~{\mapsto}~u_\mu$ the model problem,
${\hat{E}_\mu}:\mu~{\mapsto}~\hat{u}_\mu$ its lower-dimensional approximation,
$Q(u_\mu)$ the quantity of interest, and $\hat{Q}(\hat{u}_\mu)$ its RB approximation.
More specifically, the RB method consists in two steps: (i) A so-called offline stage, where solutions to $E_\mu$ for well-chosen
values of the parameter $\mu$ are computed. During this stage, $\hat{N}$ problems of size $N$ are
solved (with $\hat{N}\ll N$), and some quantities related to the $\hat{N}$ solutions are stored, and (ii) a so-called online stage, where the precomputed quantities are
used to solve $\hat{E}_\mu$ for many values of $\mu$. In this stage, a certification of the approximation is possible by means
of an a posteriori error bound.
An important feature in the RB method is the use of an online-efficient error bound.
The notion of online-efficiency is defined in Section \ref{sec:efficiency}.
Moreover, the error bound must be as sharp as possible to faithfully represent the error. However, as noticed for example
in \cite[pp.148-149]{paterabook}, the error bound is subject to round-off errors, especially for the computation of accurate solutions.
This difficulty can be encountered in complex industrial applications in the following two cases.
First and most importantly, when the stability constant of the underlying bilinear (or sesquilinear) form is very small,
the classical formula for
the error bound fails to certify, even at a relatively crude error level, as illustrated in Section \ref{sec:acoustics} where
the stability constant is about $10^{-6}$ and the classical error bound stagnates at about $10^{-4}$.
Second, in some industrial codes, the single-precision format is used to speed up computations, when high precision is not
needed. In this case, the classical formula for the error bound fails to deliver values below $10^{-4}$ for a stability
constant of order $1$.
The purpose of this work is an explanation of these facts and the derivation of a new method to compute the error bound in an
accurate and online-efficient way. Additionally, the new formula uses potentially less precomputed quantities than the classical
formula.

In Section \ref{sec:RB}, we briefly recall the main ingredients of the RB method, namely (i) the construction of the reduced problem,
(ii) the a posterior error bound, (iii) the notion of online-efficiency, and (iv) the offline stage during which the vectors of the
reduced basis are constructed.
We then explain in Section \ref{sec:machine_prec}
why the classical formula for computing the error bound is ill-conditioned in regard of round-off errors.
In Section \ref{sec:newest}, we present our new procedure based on the Empirical Interpolation Method (EIM).
A version of the EIM stabilized with respect to round-off errors is also derived, and the various procedures to compute the
error bound are compared on a simple one-dimensional diffusion problem.
In Section \ref{sec:acoustics}, we apply this new procedure to a three-dimensional acoustic scattering problem.

\section{The reduced basis method}
\label{sec:RB}
\subsection{The model problem}
\label{sec:model_pb}

We suppose that the problem of interest has the following discrete variational form,
depending on a parameter $\mu$ in a parameter set $\mathcal{P}$:
for a finite-dimensional space $\mathcal{V}$ of dimension $N$ (with $N\gg 1$ resulting, e.g., from discretization),
find $u_\mu\in \mathcal{V}$ such that 
\begin{equation}
E_\mu:a_\mu(u_\mu,v)=b(v),\qquad\forall v\in \mathcal{V},
\end{equation}
where $a_\mu$ is an inf-sup stable bounded sesquilinear form on $\mathcal{V}\times\mathcal{V}$
and $b$ is a
continuous linear form on $\mathcal{V}$. We work in complex vector spaces in view of our application to acoustic scattering.
In what follows, the complex conjugate of $z\in\mathbb{C}$ is denoted $z^*$.
We define the Riesz isomorphism $J$ from $\mathcal{V}'$ to $\mathcal{V}$ such that for all $l\in\mathcal{V}'$ and all $u\in\mathcal{V}$,
$\left(Jl,u\right)_{\mathcal{V}}=l(u)$, where $(\cdot,\cdot)_{\mathcal{V}}$ denotes the inner product of $\mathcal{V}$ with
associated norm $\|\cdot\|_{\mathcal{V}}$. We denote 
$\displaystyle{\beta_\mu}:=\underset{u\in\mathcal{V}}{\inf}\underset{v\in\mathcal{V}}{\sup}\frac{|a_{\mu}(u,v)|}{\|u\|_\mathcal{V}\|v\|_\mathcal{V}}>0$
the inf-sup constant of $a_\mu$ and
$\tilde{\beta}_\mu$ a computable positive lower bound of ${\beta_\mu}$. For simplicity,
we consider that the linear form $b$ is independent of the parameter $\mu$. The extension to $\mu$-dependent $b$ is straightforward.
We refer to the discrete solution $u_\mu$ as the ``truth solution''.

\subsection{The reduced problem}

Suppose that a reduced basis, consisting of $\hat{N}$ solutions $u_{\mu_i}$ of $E_{\mu_i}$, $i\in\{1,...,\hat{N}\}$,
has already been constructed. To alleviate the notation, we denote $u_i$ the function $u_{\mu_i}$.
How the parameters $\mu_i$ are chosen is briefly outlined in Section \ref{sec:greedy}.
Given a parameter value $\mu\in\mathcal{P}$,
the reduced problem is then a Galerkin procedure written on the linear space
$\hat{\mathcal{V}}=\textnormal{Span}\{u_1,...,u_{\hat{N}}\}\subset\mathcal{V}$: find $\hat{u}_\mu\in\hat{\mathcal{V}}$ such that
\begin{equation}
\hat{E}_\mu:a_\mu(\hat{u}_\mu,u_j)=b(u_j), \qquad\forall j\in\{1,...,\hat{N}\}.
\end{equation}
The approximate solution on the reduced basis is written as
\begin{equation}
\label{eq:defgamma}
\hat{u}_\mu=\sum_{i=1}^{\hat{N}}{\gamma_i(\mu)u_i}.
\end{equation}

Recalling the exact and approximate quantities of interest $Q(u_\mu)$ and $\hat{Q}(\hat{u}_\mu)$, respectively, the quality
of the approximation for a given $\mu\in\mathcal{P}$ is quantified by the error measure $\|Q(u_\mu)-\hat{Q}(\hat{u}_\mu)\|$.
When we obtain a satisfying error measure with $\hat{N}\ll N$, the RB strategy is successful.
Two main cases are generally considered: (i) the so-called general-purpose case, where one is interested in the whole solution:
$Q=\hat{Q}={\rm Id}$ and $\|\cdot\|=\|\cdot\|_{\mathcal{V}}$, and
(ii) the so-called goal-oriented case, where $Q$ is a linear form on $\mathcal{V}$ and $\|\cdot\|=|\cdot|$.
The operator $\hat{Q}$ is consistently built so that $\|Q(u_\mu)-\hat{Q}(\hat{u}_\mu)\|$ vanishes for $\mu=\mu_i$, $i\in\{1,...,\hat{N}\}$.

\subsection{A posteriori error bound}

In the standard RB method, the a posteriori error bound is a residual-based bound.
In what follows, we refer to it simply as error bound.
Since this error bound is an upper bound, it provides a way to certify the approximation made by the reduced basis.

\begin{property}[General-purpose case]
\label{prop_GP}
The following error bound holds: For all $\mu\in\mathcal{P}$,
\begin{equation}
\label{eq:E1}
\|u_\mu-\hat{u}_\mu\|_{\mathcal{V}}\leq \mathcal{E}_1(\mu):={\tilde{\beta}_\mu}^{-1}\|G_\mu \hat{u}_\mu\|_{\mathcal{V}},
\end{equation}
with $G_\mu$ the linear map from $\mathcal{V}$ to $\mathcal{V}$ such that
$\mathcal{V}\ni u\mapsto G_\mu u:=J\left(a_\mu(u,\cdot)-b\right)\in \mathcal{V}$. 
\end{property}

\begin{proof}
See \cite[Section 4.3.2]{paterabook}.
\end{proof}

In the goal-oriented case, one possible approach is to
introduce the following dual problem: Find $v_\mu\in\mathcal{V}$ such that
\begin{equation}
\label{eq:pb_dual}
E_{\mu}^d:a_\mu(w,v_\mu) = Q(w), \qquad\forall w\in\mathcal{V}.
\end{equation}
We wrote the dual problem on the same discrete space $\mathcal{V}$, but another space can be considered.
A reduced basis procedure is also carried out for the problem $E_{\mu}^d$, resulting in an approximation
$\hat{v}_\mu$ of $v_\mu$. The approximate quantity of interest is then defined as
$\hat{Q}(\hat{u}_\mu):=Q(\hat{u}_\mu)-(G_\mu \hat{u}_\mu,\hat{v}_\mu)_{\mathcal{V}}$, where the second term
is the so-called dual-based correction.

\begin{property}[Goal-oriented case]
The following error bound holds: For all $\mu\in\mathcal{P}$,
\begin{equation}
\label{eq:dual_corr}
\left|Q(u)-\hat{Q}(\hat{u}_\mu)\right|\leq \mathcal{E}_1^\textnormal{go}(\mu):=\left(\tilde{\beta}_\mu^d\right)^{-1}\|G_\mu\hat{u}_\mu\|_{\mathcal{V}}\|{G}_\mu^d \hat{v}_\mu\|_{\mathcal{V}},
\end{equation}
where $G_\mu^d$ is the linear map from $\mathcal{V}$ to $\mathcal{V}$ such that
$\mathcal{V}\ni v\mapsto {G}_\mu^d u:=J\left(a_\mu(\cdot,v)-Q\right)\in \mathcal{V}$
and $\tilde{\beta}_\mu^d$ is a computable lower bound of
$\displaystyle{\beta_\mu^d}=\underset{u\in\mathcal{V}}{\inf}\underset{v\in\mathcal{V}}{\sup}\frac{|a_{\mu}(v,u)|}{\|u\|_\mathcal{V}\|v\|_\mathcal{V}}$.
Obviously, ${\beta_\mu^d}={\beta_\mu}$ if $a_\mu$ is Hermitian.
\end{property}
\begin{proof}
See \cite[Proposition 23]{Boyaval} and \cite[Proposition 3.1]{cemracs}.
\end{proof}

In what follows, we mainly focus on the general-purpose case. Extensions to the goal-oriented case are straightforward.

\subsection{Online-efficiency of the RB method}
\label{sec:efficiency}

The notion of online-efficiency is central to the RB method.
\begin{definition}
The RB method is said to be online-efficient if in the online stage, (i) the reduced problems can be
constructed in complexity independent of $N$,
and (ii) the error bound can be computed in complexity independent of $N$.
\end{definition}
\begin{definition}
\label{def:affine_assump}
The sesquilinear form $a_\mu$ is said to depend on $\mu$ in an affine way if there exist $d$ functions $\alpha_k(\mu):\mathcal{P}\rightarrow\mathbb{C}$
and $d$ $\mu$-independent sesquilinear forms $a_k$ bounded on $\mathcal{V}\times\mathcal{V}$ such that
\begin{equation}
\label{eq:affine_assump}
a_\mu(u,v)=\sum_{k=1}^d{\alpha_k(\mu)a_k(u,v)}, \qquad \forall u,v \in\mathcal{V}.
\end{equation}
\end{definition}
In what follows, we always assume that the affine decomposition \eqref{eq:affine_assump} holds.
This decomposition is instrumental to achieve online-efficiency.
\begin{property}
If $a_\mu$ depends on $\mu$ in an affine way, then the RB method is online-efficient.
\end{property}
\begin{proof}
(i) The reduced matrix writes $(\hat{A}_\mu)_{j,i}=a_\mu(u_i,u_j)$ and the reduced right-hand side
$(\hat{B})_j=b(u_j)$, for all $1\leq i,j\leq \hat{N}$.
There holds $\hat{A}_\mu=\sum_{k=1}^d\alpha_k(\mu)\hat{A}_k$, where $(\hat{A}_k)_{ij}:=a_k(u_i,u_j)$.
Therefore, provided the $d$ matrices $\hat{A}_k$ and the vector $\hat{B}$ are precomputed during the offline stage, the reduced problems are constructed
in complexity independent of $N$.

(ii) The operator $G_\mu$ inherits the affine dependence of $a_\mu$ on $\mu$ since, for all $u\in\mathcal{V}$,
\begin{equation}
\label{eq:Gmu_decomp}
G_\mu u=-Jb+\sum_{k=1}^d\alpha_k(\mu)Ja_k(u,\cdot)=G_{00}+\sum_{k=1}^d\alpha_k(\mu) G_k u, 
\end{equation}
where $G_{00}:=-Jb\in\mathcal{V}$ and $G_k u:=Ja_k(u,\cdot)\in\mathcal{V}$ for all $k\in\{1,...,d\}$.
Using this affine decomposition and recalling \eqref{eq:defgamma}, we infer
\begin{equation}
\label{eq:est1}
\mathcal{E}_1(\mu)={\tilde{\beta}_\mu}^{-1}
\left\| G_{00}+\sum_{i=1}^{\hat{N}}\sum_{k=1}^{d}{\alpha_k(\mu)\gamma_i(\mu)}G_k u_i\right\|_{\mathcal{V}}. 
\end{equation}
The scalar product on which the norm in \eqref{eq:est1} hinges
can be expanded to provide another formula for the error bound (see \cite[eq.(4.61)]{paterabook}):
\begin{equation}
\label{eq:E2}
\begin{aligned}
\displaystyle {\mathcal{E}_2(\mu)} ={}& {\tilde{\beta}_\mu}^{-1}\left({(G_{00},G_{00})_{\mathcal{V}}}
\displaystyle +2{\rm Re}\sum_{i=1}^{\hat{N}}\sum_{k=1}^{d}{\gamma_i(\mu)\alpha_k(\mu){(G_k u_i,G_{00})_{\mathcal{V}}}}\right.\\
\displaystyle &\left.+\sum_{i,j=1}^{\hat{N}}\sum_{k,l=1}^{d}{\gamma_i(\mu)\alpha_k(\mu)\gamma_j^*(\mu)\alpha_l^*(\mu){(G_k u_i,G_l u_j)_{\mathcal{V}}}}\right)^{\frac{1}{2}},
\end{aligned}
\end{equation}
which is computed in complexity independent of $N$ in the online stage provided that $(G_{00},G_{00})_{\mathcal{V}}$, $(G_k u_i,G_{00})_{\mathcal{V}}$
and $(G_k u_i,G_l u_j)_{\mathcal{V}}$ are precomputed during the offline stage, and provided that a
lower bound ${\tilde{\beta}_\mu}$ of the stability constant
of $a_\mu$ is also computed in complexity independent of $N$ (which is possible, for example, by the Successive Constraint
Method, see \cite{Huynh, Chen}).
\end{proof}

An important observation made in \cite{casenave}, and that will be useful below, is that
the formula \eqref{eq:E2} defining $\mathcal{E}_2$ can be rewritten in an equivalent way as
\begin{equation}
\label{eq:est2}
\mathcal{E}_2(\mu):={\tilde{\beta}_\mu}^{-1}\left(\delta^2 + 2{\rm Re} (s^t \hat{x}_\mu) + {\hat{x}_\mu}^{*t} S\hat{x}_\mu\right)^{\frac{1}{2}},
\end{equation}
where $\delta:=\|G_{00}\|_{\mathcal{V}}$, $s$ and $\hat{x}_\mu$ are vectors in $\mathbb{C}^{d\hat{N}}$ with components
$s_I:=(G_k u_i,G_{00})_{\mathcal{V}}$ and $(\hat{x}_{\mu})_I:=
\alpha_k(\mu)\gamma_i(\mu)$, and $S$ is a matrix in $\mathbb{C}^{d\hat{N},d\hat{N}}$ with coefficients
$S_{I,J}:=(G_k u_i,G_l u_j)_{\mathcal{V}}$ (with $I$ and $J$ re-indexing respectively $(k,i)$ and $(l,j)$,
for all $1\leq k,l\leq d$ and all $1\leq i,j\leq\hat{N}$). The $t$ superscript denotes the transposition. The vector $s$ and the matrix $S$
depend on the reduced basis functions $\{u_i\}_{1\leq i\leq\hat{N}}$ but are independent of $\mu$,
and the vector $\hat{x}_\mu$ depends on the RB approximation $\hat{u}_\mu$ via the coefficients $\gamma_i(\mu)$.
Notice that the term between parenthesis on the right-hand side of \eqref{eq:est2} is a multivariate polynomial in $\hat{x}_{\mu}$
of total degree $2$.
We would like to stress that $\mathcal{E}_1(\mu)=\mathcal{E}_2(\mu)$ (in infinite precision arithmetic): the indices $1$ and $2$ are used to
denote two different ways to compute the same quantity. In particular, $\mathcal{E}_1(\mu)$ is not online efficient, while $\mathcal{E}_2(\mu)$
is.

\subsection{The offline stage}
\label{sec:greedy}

Fix a discrete subset of parameters $\mathcal{P}_{\rm trial}\subset\mathcal{P}$.
In the offline stage, the parameters $\mu_i$ (from which the reduced basis is constructed)
are chosen by a greedy algorithm as elements of $\mathcal{P}_{\rm trial}$.
We denote $\mathcal{P}_{\rm select}$ the set of these selected parameters;
see \cite[Section 3.3]{paterabook} for a presentation of the greedy algorithm.
At each step of the algorithm, the new quantities $a_k(u_i,u_j)$ and $b(u_j)$ are computed and stored,
as well as the new components of the vector $s$ and of the matrix $S$ to be used in the formula \eqref{eq:est2} for $\mathcal{E}_2$.
This task, as that of evaluating $G_{00}$, typically requires inverting the stiffness matrix in $\mathcal{V}$ by solving, for
all $k\in\{1,...,d\}$ and all $i\in\{1,...,\hat{N}\}$, the variational problem: find $w_{i,k}\in\mathcal{V}$ such that
\begin{equation}
\label{eq:compute_G}
{E_G}_{i,k}:(w_{i,k},v)_{\mathcal{V}}=a_k(u_i,v),\qquad\forall v\in\mathcal{V}.
\end{equation}
Then, $G_k u_i = w_{i,k}$ can be computed.
The computation of $(G_k u_i,G_l u_j)_{\mathcal{V}}$ follows from the solutions of ${E_G}_{i,k}$ and ${E_G}_{j,l}$.
Since the error bounds are evaluated using the formula $\mathcal{E}_{2}(\mu)$, for all $\mu\in\mathcal{P}_{\rm trial}$,
with the current state of the reduced basis, finding the maximum of the error bound on
$\mathcal{P}_{\rm trial}$ is of complexity independent of $N$. This allows one to consider very large sets $\mathcal{P}_{\rm trial}$
without increasing too much the complexity of the whole offline procedure.

\section{Round-off errors and online certification}
\label{sec:machine_prec}
In this section, we explain why the online-efficient error bound \eqref{eq:est2} may be sensitive to round-off errors.

\subsection{Elements of floating-point arithmetic}
\label{sec:intro_precision}

In a computer, real numbers are represented by a finite number of bits, called floating-point representation.
Current architectures are optimized for a format
used by a large majority of softwares: IEEE 754 double-precision binary floating-point format.
Let $x$ be a real number. The floating point representation of $x$ is denoted by $fl(x)$.
When a (nonzero) real number is rounded to the closest floating-point number, the relative error on its floating-point
representation is bounded by a number, $\epsilon$, called the machine precision.
In double precision, $\epsilon=5\times 10^{-16}$ (see \cite[Section 1.2]{Goldberg91}).
Let $x$ and $y$ be real numbers. When computing the operation $x+y$, the result returned by the computer can be different
from its theoretical value. Whenever the difference is substantial, a loss of significance occurs.
A well-known case of loss of significance is when $x$ and $y$ are almost opposite numbers. Suppose that $x=-y$.
We denote by ${\rm maxfl}(x+y)$ the result that the computer returns
when the maximal accumulation of round-off errors occurs when computing the summation. There holds
\begin{equation}
\label{eq:maxfl}
|{\rm maxfl}(x+y)|\approx 2\epsilon|x|.
\end{equation}

When implementing an algorithm, one should ensure that each step is free of such a loss of significance.
In some cases, simply changing the order of the operations can prevent these situations.
As an illustration, consider $x=1$, $y=1+10^{-7}$, and the operation $x^2-2xy+y^2$. This is a sum of terms where the first
intermediate result in the sum is $14$ orders larger than the result. Therefore, a loss of significance is expected.
The relative error of this computation is about $8\times 10^{-4}$. Computing $(x-y)^2$, which is the
factorization of the considered operation, leads to a relative error of about $10^{-9}$. Thus, the terms of the sum
are only $7$ orders larger than the results, leading to a less catastrophic loss of significance. In this specific case,
the remedy consists in carrying out the sum before the multiplication.
In the RB context, the evaluation of the formula $\mathcal{E}_2$ suffers from such a loss of significance,
as we now explain.

\subsection{Validity of the formulae $\mathcal{E}_1$ and $\mathcal{E}_2$ for computing the error bound}

Consider the two formulae $\mathcal{E}_1$, see \eqref{eq:est1}, and $\mathcal{E}_2$, see \eqref{eq:est2}, for computing the error bound.
\begin{definition}
\label{def_val}
The formula $\mathcal{E}_k$, $k=1,2$, is said to be valid for computing the error bound with tolerance ${\rm tol}$ if
\begin{equation}
\underset{\mu\in\mathcal{P}_{\rm select}}{\max}(\mathcal{E}_k(\mu))\leq\textnormal{tol}.
\end{equation}
\end{definition}

From a theoretical viewpoint, the error $\|u_\mu-\hat{u}_\mu\|_{\mathcal{V}}$ and the residual $G_{\mu}u_{\mu}$ vanish for all
$\mu\in\mathcal{P}_{\rm select}$.
Hence, any formula for computing the residual-based error bound vanishes as well and therefore is valid with any tolerance.
However, the validity of a formula for computing the error bound is to be considered
in the presence of some adverse phenomenon introducing errors in the computation, see Figure \ref{fig:def_valid}.
The greedy algorithm in the offline stage stops when
$\underset{\mu\in\mathcal{P}_{\rm trial}}{\max}(\mathcal{E}_k(\mu))<{\rm tol_{RB}}$, where ${\rm tol_{RB}}$ denotes the
maximum acceptable error made by the RB approximation.
Therefore, if the minimum tolerance for which an error bound $\mathcal{E}_k$ is valid is larger than ${\rm tol_{RB}}$,
then the greedy algorithm cannot converge and will keep increasing the set $\mathcal{P}_{\rm select}$ although the error can be actually
very small.

\begin{figure}[h!]
\centerline{
\includegraphics [width=8cm] {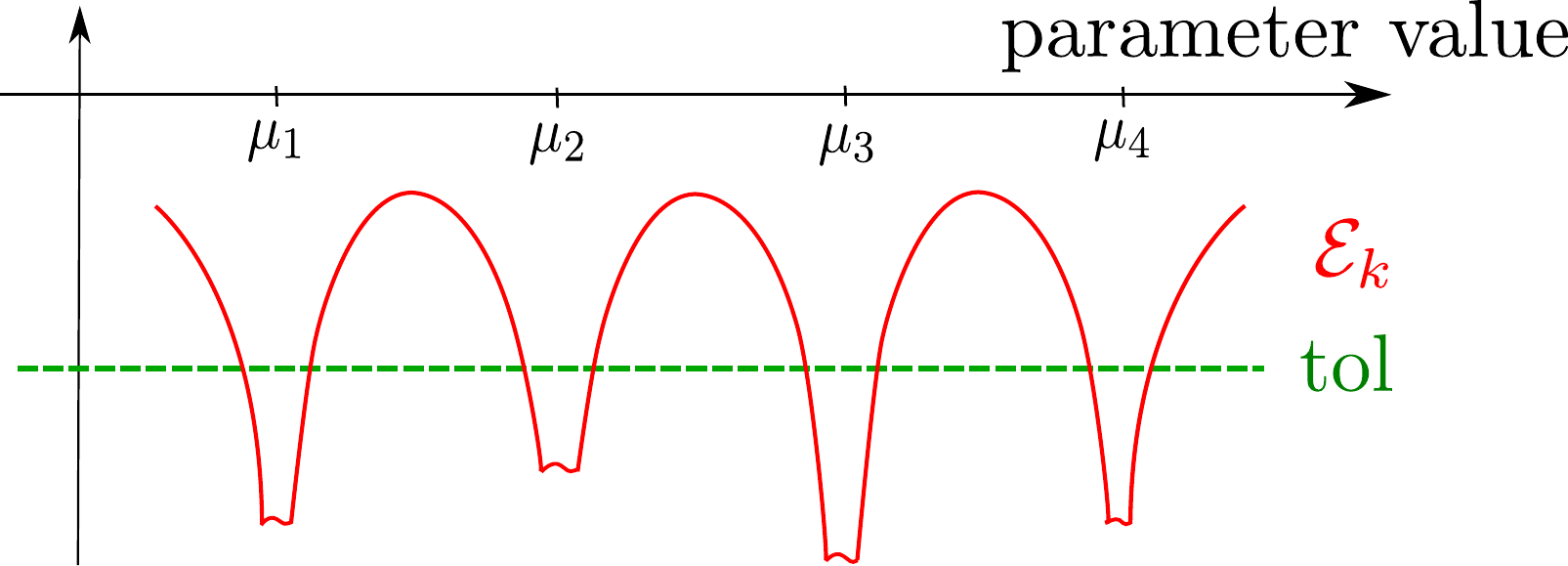}
\includegraphics [width=8cm] {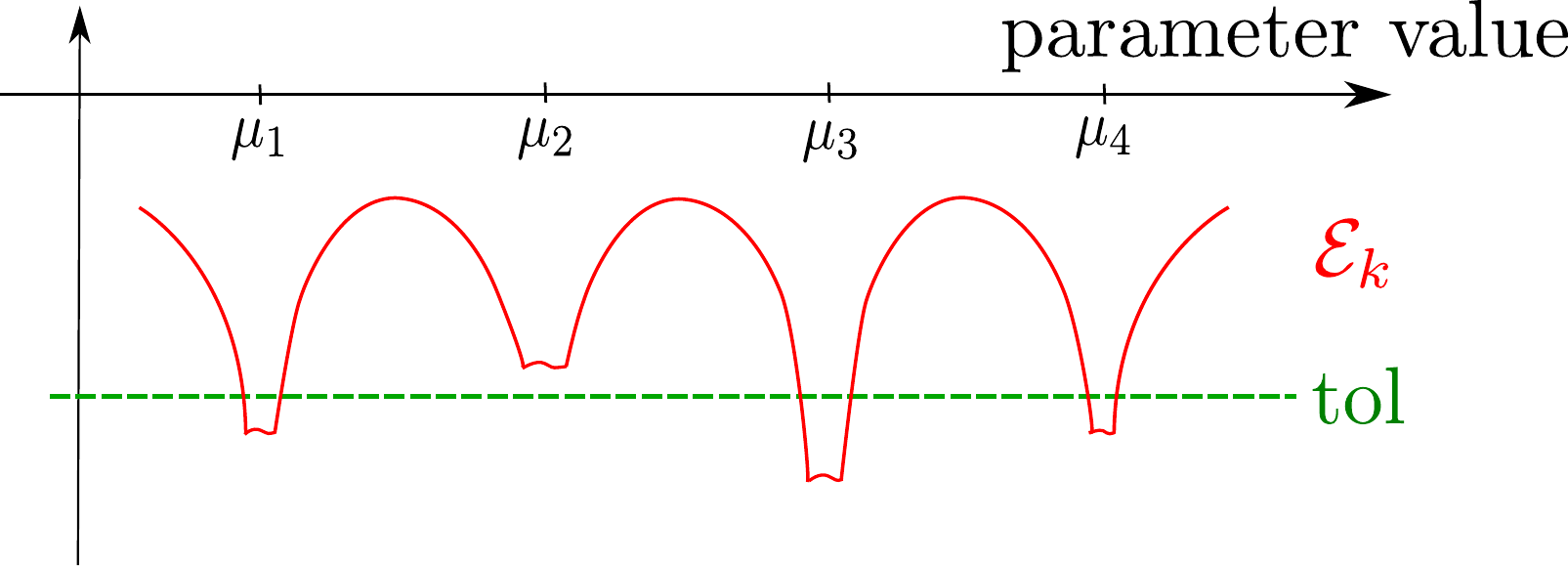}}
\caption{Schematic illustration of Definition \ref{def_val},
with $\mathcal{P}_{\rm select}=\{\mu_1,...,\mu_4\}$. Left: the formula $\mathcal{E}_k$ is valid for computing
the error bound with tolerance $\rm{tol}$ ; right: the formula is not valid as $\mathcal{E}_k(\mu_2)>{\rm tol}$.}
\label{fig:def_valid}
\end{figure}

We examine the validity of the formulae $\mathcal{E}_1$ and $\mathcal{E}_2$ for computing the error bound in the
presence of two independent phenomena: round-off errors and approximate reduced basis functions $u_i$ (in the context of inexact
linear algebra solvers for $E_{\mu_i}$).

\subsubsection{Round-off errors}

We investigate the influence of round-off errors when computing the error bounds $\mathcal{E}_1(\mu)$
and $\mathcal{E}_2(\mu)$. As observed at the end of Section \ref{sec:intro_precision},
the computation of a polynomial using a factorized form is more accurate than using the developed form, in particular at points
close to its roots. Here, $\left({{\tilde{\beta}_\mu}}\mathcal{E}_2(\mu)\right)^2$ is a multivariate polynomial of degree
$2$ in ${\hat{x}_\mu}$ computed in a developed form, whereas the scalar product $(G_\mu u_\mu,G_\mu u_\mu)_{\mathcal{V}}$
used in the computation of $\mathcal{E}_1(\mu)$ is not developed.

In this section, we neglect the round-off errors introduced when solving $E_\mu$ and $\hat{E}_\mu$, so that
the reduced basis functions $u_i$ and the reduced solutions $\hat{u}_\mu$ are considered free of round-off errors.
We also suppose that the computable positive lower bound $\tilde{\beta}_\mu$ of the inf-sup constant
is computed free of round-off errors, see Remark \ref{remarkinfsupcte}.

\begin{proposition}
Let $\mu\in\mathcal{P}_{\rm select}$ and let $\rm{maxfl}(\tilde{\beta}_\mu \mathcal{E}_k(\mu))$, $k=1,2$,
denote the evaluation of $\tilde{\beta}_\mu \mathcal{E}_k(\mu)$ when the maximum accumulation of round-off errors occurs. There holds
\begin{equation}
\label{eq:lower_bounds}
\begin{aligned}
&\rm{maxfl}(\tilde{\beta}_\mu \mathcal{E}_1(\mu))\geq 2\delta\epsilon,\\
&\rm{maxfl}(\tilde{\beta}_\mu\mathcal{E}_2(\mu))\geq 2\delta\sqrt{\epsilon},
\end{aligned}
\end{equation}
where $\delta=\|G_{00}\|_{\mathcal{V}}$ and $\epsilon$ is the machine precision.
\end{proposition}
\begin{proof}
Let $\mu\in\mathcal{P}_{\rm select}$. We present the proof for $\mathcal{E}_1(\mu)$; the proof for $\mathcal{E}_2(\mu)$
is similar. We need to evaluate the right-hand side of \eqref{eq:est1}.
Let $(\varphi_\rho)_{1\leq \rho\leq N}$ denote the basis of $\mathcal{V}$, so that, for instance,
$G_{00}=\sum_{\rho=1}^N \left(G_{00}\right)_{\rho}\varphi_{\rho}$.
In exact arithmetics, there holds $\mathcal{E}_1(\mu)=0$, so that 
$\sum_{i=1}^{\hat{N}}\sum_{k=1}^d\gamma_i(\mu)\alpha_k({\mu}) \left(G_k u_i\right)_{\rho}=-\left(G_{00}\right)_{\rho}$
for all $1\leq \rho\leq N$. As a result, using \eqref{eq:maxfl}, we obtain
\begin{equation*}
\left|\rm{maxfl}\left((G_{00})_{\rho}+\sum_{i=1}^{\hat{N}}\sum_{k=1}^d\gamma_i(\mu)\alpha_k({\mu}) (G_k u_i)_{\rho}\right)\right|
\approx 2|(G_{00})_{\rho}|\epsilon.
\end{equation*}
Since computing the $\mathcal{V}$-norm on the right-hand side of \eqref{eq:est1} can only increase the round-off errors,
we infer the desired lower bound.
\end{proof}

\begin{remark}[Validity of the formulae $\mathcal{E}_1$ and $\mathcal{E}_2$]
\label{direct_case}
We indeed observe in our simulations that the round-off errors on $\mathcal{E}_1$ scale like $\epsilon$,
while the round-off errors on $\mathcal{E}_2$ scale like $\sqrt{\epsilon}$ (see Section \ref{sec:numerics}).
Then, if we suppose that the lower bounds are reached in \eqref{eq:lower_bounds},
the formulae $\mathcal{E}_1$ and $\mathcal{E}_2$ are valid for computing the error bound with
tolerance $\rm{tol}$ if, respectively,
\begin{equation}
\begin{alignedat}{2}
&\textnormal{for }\mathcal{E}_1,&\qquad 2\left(\tilde{\beta}_{\rm min}\right)^{-1}\delta\epsilon&\leq\textnormal{tol},\\
&\textnormal{for }\mathcal{E}_2,&\qquad 2\left(\tilde{\beta}_{\rm min}\right)^{-1}\delta\sqrt{\epsilon}&\leq\textnormal{tol},
\end{alignedat}
\end{equation}
where $\tilde{\beta}_{\rm min}=\underset{\mu\in\mathcal{P}_{\rm select}}{\inf}(\tilde{\beta}_{\mu})$.
\end{remark}

\begin{remark}[Inf-sup constant]
\label{remarkinfsupcte}
The computable positive lower bound $\tilde{\beta}_\mu$ of the inf-sup constant suffers from round-off errors as well.
However, since it is a multiplicative factor, the quality of its computation does not severely affect
the quality of the error bound.
Moreover, the value of the inf-sup constant does not depend on the size of the reduced basis, contrary to
$\|G_\mu \hat{u}_\mu\|_{\mathcal{V}}$.
Therefore, there is no phenomenon susceptible to degrade the accuracy of its computation with the increase of the size
of the reduced basis.
If the Successive Constraint Method is used,
the procedure to compute $\tilde{\beta}_\mu$ is carried out before the greedy algorithm of the RB method.
\end{remark}

\begin{remark}[Improved floating-point arithmetic]
Increasing the machine precision from $\epsilon$ to $\epsilon^2$ (quadruple-precision) for computing the coefficients
in \eqref{eq:est2}, as well as for the evaluation of the multivariate polynomial in ${\hat{x}_\mu}$, is
a first solution to recover a good precision with the formula $\mathcal{E}_2$.
There are also methods allowing one to double the precision of the evaluation of a polynomial while keeping the double-precision format,
namely compensated schemes. For instance, the compensated Horner scheme in double-precision \cite{Langlois}
doubles the precision and is faster than the full quadruple precision implementation. However, this corresponds to representing
the result of the intermediate operations by two doubles, one for the value in double-precision and another one for the
subsequent digits. Therefore, these strategies are equivalent to quadruple precision
(except for the computational savings in evaluating the error bound). Moreover, since
current architectures are optimized for the double-precision format, changing the floating-point arithmetic can potentially
degrade software performance.
\end{remark}

\begin{remark}[Goal-oriented case, round-off errors]
The same analysis can be carried-out in the goal-oriented case. Let $\mu\in\mathcal{P}_{\rm select}$.
There holds
\begin{equation}
\label{eq:lower_bounds2}
\begin{aligned}
&\textnormal{maxfl}(\tilde{\beta}_{\mu}^d\mathcal{E}_1^{\rm go}(\mu))\geq 2\delta\varsigma\epsilon^2,\\
&\textnormal{maxfl}(\tilde{\beta}_{\mu}^d\mathcal{E}_2^{\rm go}(\mu))\geq 2\delta\varsigma\epsilon,
\end{aligned}
\end{equation}
where $\varsigma:=\|Q\|_{\mathcal{V}'}$.
We indeed observe in our simulations that the round-off errors on $\mathcal{E}_1^{\rm go}$ scale
like $\epsilon^2$, while the round-off errors on $\mathcal{E}_2^{\rm go}$ scale like $\epsilon$ (see Section \ref{sec:acoustics}).
If we suppose that the lower bounds are reached in \eqref{eq:lower_bounds2}, then
the formulae $\mathcal{E}_1^{\rm go}$ and $\mathcal{E}_2^{\rm go}$ are valid for computing the error bound
with tolerance $\rm{tol}$ if, respectively,
\begin{equation}
\begin{alignedat}{2}
&\textnormal{for }\mathcal{E}_1^{\rm go},&\qquad 2\left(\tilde{\beta}_{\rm min}^d\right)^{-1}\delta\varsigma\epsilon^2&\leq\textnormal{tol},\\
&\textnormal{for }\mathcal{E}_2^{\rm go},&\qquad 2\left(\tilde{\beta}_{\rm min}^d\right)^{-1}\delta\varsigma\epsilon&\leq\textnormal{tol},
\end{alignedat}
\end{equation}
where $\tilde{\beta}_{\rm min}^d=\underset{\mu\in\mathcal{P}_{\rm select}}{\inf}(\tilde{\beta}_{\mu}^d)$.
\end{remark}

\subsubsection{Approximate reduced basis functions}

In large-scale simulations, the accuracy of the RB procedure is also limited by the numerical method used for computing
the reduced basis functions. We want here to illustrate this fact on a simple example where we suppose
that the approximation of the reduced basis functions comes from an iterative solver
with prescribed stopping criterion.
We recall that for a given value $\mu\in\mathcal{P}_{\rm select}$, $E_\mu$ consists in solving a linear system of size $N$ 
of the form $A_\mu U_\mu = B$. Thus, for $\mu\in\mathcal{P}_{\rm trial}$, the formulae $\mathcal{E}_1$
and $\mathcal{E}_2$ for the
error bound are based on the computation of the residual of $E_\mu$ for the reduced solution $\hat{u}_\mu$. Indeed, it is
easy to see that
$\|G_\mu \hat{u}_\mu\|_{\mathcal{V}}=\|A_\mu \hat{U}_\mu - B\|_{*\mathcal{V}'}$, where for all
$\Phi\in\mathbb{C}^{N}$, $\|\Phi\|_{*\mathcal{V}'}=\underset{V\in\mathbb{C}^{N}}{\sup}\frac{\left|(V,\Phi)_{\mathbb{C}^{N}}\right|}
{\|\sum_{i=1}^{N}V_i\varphi_i\|_{\mathcal{V}}}$, recalling that $(\varphi_\rho)_{1\leq \rho\leq N}$ are
the basis functions in $\mathcal{V}$, see \cite[\S 9.1.5]{ern}.

In this section, we suppose that the formulae
$\mathcal{E}_1$ and $\mathcal{E}_2$ are free of round-off errors (therefore, for all $\mu\in\mathcal{P}_{\rm trial}$,
$\mathcal{E}_1(\mu)=\mathcal{E}_2(\mu)$), but the problem $E_\mu$ is not solved exactly, leading to
approximate reduced basis functions such that the residuals do not vanish.
Hence, for all $\mu\in\mathcal{P}_{\rm select}$, $\mathcal{E}_{1}(\mu)=\mathcal{E}_{2}(\mu)$ and these error bounds are
nonzero owing to inexact linear algebra solves.
The reduced problems $\hat{E}_\mu$ are supposed to be solved freely of round-off errors.

\begin{proposition}[Approximate reduced basis functions]
If the reduced basis functions are computed using an iterative solver with the following stopping criterion on the normalized residual:
\begin{equation}
\label{eq:stopcrit0}
\forall\mu\in\mathcal{P}_{\rm trial},\qquad\frac{\|A_{\mu}U_\mu-B\|_{*\mathcal{V}'}}{\|B\|_{*\mathcal{V}'}}\leq \xi,
\end{equation}
then the formulae $\mathcal{E}_1$ and $\mathcal{E}_2$ are valid for computing the error bound
with tolerance $\rm{tol}$ if
\begin{equation}
\tilde{\beta}_{\rm min}^{-1}\delta\xi\leq\textnormal{tol}.
\end{equation}
\end{proposition}
\begin{proof}
Let $k\in\{1,2\}$, let $\mu\in\mathcal{P}_{\rm select}$ and suppose that the stopping criterion \eqref{eq:stopcrit0} is
satisfied. Then, $\hat{u}_{\mu}={u}_{\mu}$, but ${u}_{\mu}$ does not exactly solve $E_{\mu}$.
First, by definition of the $\|\cdot\|_{*\mathcal{V}}$ norm, $\|B\|_{*\mathcal{V}'}=\underset{V\in\mathbb{C}^{N}}{\sup}\frac{\left|b(\sum_{i=1}^{N}V_i\varphi_i)\right|}
{\|\sum_{i=1}^{N}V_i\varphi_i\|_{\mathcal{V}}}=\|b\|_{\mathcal{V}'}=\|G_{00}\|_{\mathcal{V}}=\delta$. Then,
$\|G_\mu \hat{u}_\mu\|_{\mathcal{V}}=\underset{v\in\mathcal{V}}{\sup}\frac{(G_\mu \hat{u}_\mu,v)_{\mathcal{V}}}{\|v\|_{\mathcal{V}}}
=\underset{v\in\mathcal{V}}{\sup}\frac{a_\mu(\hat{u}_\mu,v)-b(v)}{\|v\|_{\mathcal{V}}}=\underset{V\in\mathbb{C}^{N}}{\sup}\frac{(V,A_\mu \hat{U}_\mu-B)_{\mathbb{C}^{N}}}
{\|\sum_{i=1}^{N}V_i\varphi_i\|_{\mathcal{V}}}=\|A_{\mu}\hat{U}_\mu-B\|_{*\mathcal{V}'}$.
Therefore, 
\begin{equation*}
\mathcal{E}_k(\mu)
=\tilde{\beta}_{\mu}^{-1}\|G_\mu \hat{u}_\mu\|_{\mathcal{V}}=\tilde{\beta}_{\mu}^{-1}\|A_{\mu}\hat{U}_\mu-B\|_{*\mathcal{V}'}=\tilde{\beta}_{\mu}^{-1}\|A_{\mu}U_\mu-B\|_{*\mathcal{V}'}\leq
\tilde{\beta}_{\mu}^{-1}\|B\|_{*\mathcal{V}'}\xi=\tilde{\beta}_{\mu}^{-1}\delta\xi\leq\tilde{\beta}_{\rm min}^{-1}\delta\xi.
\end{equation*}
Hence, if $\tilde{\beta}_{\rm min}^{-1}\delta\xi\leq \rm{tol}$, the validity of $\mathcal{E}_1$ and $\mathcal{E}_2$ follows
from Definition \ref{def_val}.
\end{proof}

Since the $\|\cdot\|_{*\mathcal{V}'}$ norm is hard to compute, the stopping criterion \eqref{eq:stopcrit0} uses
in practice the Hermitian norm in $\mathbb{C}^N$ or the $\mathcal{V}$-norm of the corresponding functions in $\mathcal{V}$.

\begin{remark}[Goal-oriented case, approximate reduced basis functions]
The formulae $\mathcal{E}_1^{\rm go}$ and $\mathcal{E}_2^{\rm go}$ are valid for computing the error bound with tolerance
$\rm{tol}$ if $\left(\tilde{\beta}_{\rm min}^d\right)^{-1}\delta\gamma\xi^2\leq\textnormal{tol}$.
\end{remark}

\subsubsection{Synthesis}

Taking into account the round-off errors in the computation of the error bound and the stopping criterion
of an iterative solver, and supposing that the bounds \eqref{eq:lower_bounds} and \eqref{eq:lower_bounds2} are reached,
the formulae $\mathcal{E}_1$ and $\mathcal{E}_2$ are valid
for computing the error bound with tolerance $\rm{tol}$ if, respectively,
\begin{equation}
\begin{aligned}
&\textnormal{for }\mathcal{E}_1,\qquad2\tilde{\beta}_{\rm min}^{-1}\delta\max\left(\xi,\epsilon\right)\leq\textnormal{tol},\\
&\textnormal{for }\mathcal{E}_2,\qquad2\tilde{\beta}_{\rm min}^{-1}\delta\max\left(\xi,\sqrt{\epsilon}\right)\leq\textnormal{tol},
\end{aligned}
\end{equation}
and the formulae $\mathcal{E}_1^{\rm go}$ and $\mathcal{E}_2^{\rm go}$ are valid
for computing the error bound with tolerance $\rm{tol}$ if, respectively,
\begin{equation}
\begin{aligned}
&\textnormal{for }\mathcal{E}_1^{\rm go},\qquad2\left(\tilde{\beta}_{\rm min}^d\right)^{-1}\delta\gamma\max\left(\xi^2,\epsilon^2\right)\leq\textnormal{tol},\\
&\textnormal{for }\mathcal{E}_2^{\rm go},\qquad2\left(\tilde{\beta}_{\rm min}^d\right)^{-1}\delta\gamma\max\left(\xi^2,\epsilon\right)\leq\textnormal{tol}.
\end{aligned}
\end{equation}

Focusing on round-off errors, the formula $\mathcal{E}_1$ for computing the error bound is valid for tolerances scaling as $\epsilon$, but is not
online-efficient, whereas the formula $\mathcal{E}_2$ is online-efficient but is valid only for (significantly) higher tolerances, namely
tolerances scaling as $\sqrt{\epsilon}$.

\section{New procedures for accurate and efficient evaluation of the error estimator}
\label{sec:newest}
In this section, online-efficient methods, that are valid for tolerances scaling as $\epsilon$, are devised to evaluate the error bound.

\subsection{Procedure 1: rewriting $\mathcal{E}_2$}
\label{sec:first_rewriting}

We first present the procedure proposed in \cite{casenave}. We consider that a reduced basis of size $\hat{N}$ has been constructed. Let 
$\sigma:=1+2d\hat{N}+(d\hat{N})^2$.
For a given $\mu\in\mathcal{P}_{\rm trial}$ and the resulting $\hat{u}_\mu\in \textnormal{Span}\{u_1, ..., u_{\hat{N}}\}$
solving the reduced problem, we define $\hat{X}(\mu)\in\mathbb{C}^{\sigma}$
as the vector with components $(1,{\hat{x}}_{\mu_I},{\hat{x}^*}_{\mu_I},{\hat{x}^*}_{\mu_I} {\hat{x}}_{\mu_J})$, where ${\hat{x}}_{\mu_I}=\alpha_k(\mu)\gamma_i(\mu)$
% 
% $(1,{\hat{x}}_{\mu_I},{\hat{x}}_{\mu_I} {\hat{x}}_{\mu_J})$, where ${x_\mu}_I=\alpha_k(\mu)\gamma_i(\mu)$
(we recall that $\gamma_i(\mu)$ are the coefficients of the reduced solution in the reduced basis, see \eqref{eq:defgamma}, and
$\alpha_k(\mu)$ the coefficients of the affine decomposition of $a_\mu$ in \eqref{eq:affine_assump}),
with $1\leq I,J \leq d\hat{N}$ (with $I=i+\hat{N}(k-1)$ such that $1\leq i\leq \hat{N}$, $1\leq k\leq d$, and
with $J=j+\hat{N}(l-1)$ such that $1\leq j\leq \hat{N}$, $1\leq l\leq d$).
We can write the right-hand side of \eqref{eq:est2} as a linear form in $\hat{X}(\mu)$ as follows:
\begin{equation}
\label{eq:linearform}
\displaystyle \delta^2 + 2{\rm Re}(s^t \hat{x}_\mu) + {\hat{x}_\mu}^{*t} S\hat{x}_\mu = \sum_{p=1}^{\sigma}{t_p \hat{X}_p(\mu)},
\end{equation}
where $t_p$ is independent of $\mu$ (as $\delta$, $s$, and $S$ are independent of $\mu$)
and $\hat{X}_p(\mu)$ is the $p$-th component of $\hat{X}(\mu)$.

Now, in the offline stage, we take $\sigma$ values (e.g. random values)
$\mu_r\in\mathcal{P}_{\rm trial}$, $r\in\{1,...,\sigma\}$, of the parameter $\mu$. Then, we compute the
vectors $\hat{X}(\mu_r)$ and the quantities
\begin{equation}
\label{eq:V_r}
V_r:=\sum_{p=1}^{\sigma}{t_p \hat{X}_p(\mu_r)}.
\end{equation}
Finally, we define $T\in\mathbb{C}^{\sigma\times\sigma}$ as the matrix whose columns are formed by the vectors $\hat{X}(\mu_r)$,
that is, $T_{pr}=\hat{X}_p(\mu_r)$ for all $1\leq p,r\leq\sigma$. We assume
that $T$ is invertible, which always happens to be the case in our simulations.

Now, suppose that in the online stage we want to evaluate the error bound for the RB solution $\hat{u}_\mu$ computed at
a certain parameter $\mu\in\mathcal{P}_{\rm trial}$. Then, we evaluate the vector $\hat{X}(\mu)$ and solve the linear system
\begin{equation}
\label{eq:sysT}
\displaystyle T\lambda(\mu)=\hat{X}(\mu),
\end{equation}
yielding $\lambda(\mu)\in\mathbb{C}^\sigma$. We then obtain
$\hat{X}(\mu) = \sum_{r=1}^{\sigma}{\lambda_r(\mu) \hat{X}(\mu_r)}$ and
\begin{equation}
\sum_{p=1}^\sigma t_p \hat{X}_p(\mu) = 
\sum_{p,r=1}^\sigma t_p \lambda_r(\mu) \hat{X}_p(\mu_r) =
\sum_{r=1}^\sigma{\lambda_r(\mu) V_r}.
\end{equation}
This yields the following new formula for computing the error bound:
\begin{equation}
\label{eq:Est3}
\mathcal{E}_3(\mu):={\tilde{\beta}_{\mu}}^{-1}\left(\sum_{r=1}^\sigma{\lambda_r(\mu) V_r}\right)^{\frac{1}{2}},
\end{equation}
where the quantities $V_r=\left\| G_{\mu_r} \hat{u}_{\mu_r}\right\|^2_{\mathcal{V}}$ can be precomputed.
Thus, computing $\mathcal{E}_3$ requires solving \eqref{eq:sysT} and summing the $\sigma$ precomputed quantities $V_r$.
Since the complexity of this procedure is independent of $N$, the formula $\mathcal{E}_3$ is online-efficient for
computing the error bound.

\begin{remark}[Goal-oriented case]
For the goal-oriented case, the procedure is carried out independently on the two multivariate polynomials
$\|G_\mu\hat{u}_\mu\|^2_{\mathcal{V}}$ and $\|{G}_\mu^d \hat{v}_\mu\|^2_{\mathcal{V}}$.
\end{remark}

Notice that $\mathcal{E}_1(\mu)$, $\mathcal{E}_2(\mu)$, and $\mathcal{E}_3(\mu)$ are equal in exact arithmetic.
As pointed out in \cite{casenave}, the matrix $T$ exhibits in practice large condition numbers, and there is no guarantee that
$T$ is actually invertible. We will see in Section \ref{sec:acoustics} for a three-dimensional acoustic scattering
problem that $\mathcal{E}_3$ can be in practice as ill-behaved as $\mathcal{E}_2$. Moreover, there is no a priori
method for selecting the parameters $\mu_r$ for which the quantities $V_r$ are precomputed. In the next section, we propose
a new procedure that solves these problems.

\subsection{Procedure 2: improvement on Procedure 1 using the EIM}
\label{sec:principle}

In the formula $\mathcal{E}_3$, a potentially ill-conditioned problem $T\lambda(\mu)=\hat{X}(\mu)$ is solved in order to exactly
represent $\hat{X}(\mu)$ by the linear combination $\sum_{r=1}^{\sigma}{\lambda_r(\mu) \hat{X}(\mu_r)}$.
Following a suggestion by Patera \cite{Pateracom}, we propose to approximate $\hat{X}(\mu)$ by means of an interpolation procedure.
We want to modify the formula $\mathcal{E}_3$ by an interpolation formula relying on a better conditioned linear system.
The price to pay is that the new formula $\mathcal{E}_4$ will not be equal to $\mathcal{E}_1$ in exact arithmetic;
the interpolation errors are however marginal, as further discussed in Remark \ref{rmkerroracc}.
We also look for a way to choose the
parameters $\mu_r$ for which the quantities $V_r$ have to be precomputed. We refer to these values for $\mu_r$ as ``interpolation points'',
and to the set of these points as $\mathcal{P}_{\rm inter}$.

Consider the function of two variables $(p,\mu)\mapsto \hat{X}_p(\mu)$, for all $p\in\{1,...,\sigma\}$ and all 
$\mu\in\mathcal{P}_{\rm trial}$. We look for an approximation of this function in the form
\begin{equation}
\label{eq:interpb}
\forall\mu\in\mathcal{P_{\rm trial}},\forall p\in\{1, ..., \sigma\},~\hat{X}_p(\mu) \approx \sum_{r=1}^{\hat{\sigma}}{\lambda_r^{\hat{\sigma}}(\mu) \hat{X}_p(\mu_r)},
\end{equation}
for a certain parameter $\hat{\sigma}\leq \sigma$. The empirical interpolation method (EIM) (more precisely the discrete EIM since $p$ is a discrete
variable) provides a
numerical procedure to construct this approximation and to choose the interpolation points (see \cite{Barrault, Maday}).

For completeness, we briefly describe the EIM and adapt the notation of \cite{Maday} to the present context.
The EIM is an offline-online procedure. During the offline stage, $\hat{\sigma}$ basis functions are computed,
denoted $q_j:\mathcal{P}_{\rm trial}\ni\mu\mapsto q_j(\mu)\in\mathbb{C}$, for all $j\in\{1, ...,\hat{\sigma}\}$. These basis functions will be
used in the online stage to carry out the interpolation. We define $q^{\hat{\sigma}}$ as the vector-valued map
$\mathcal{P}_{\rm trial}\ni\mu\mapsto q^{\hat{\sigma}}(\mu):=(q_j(\mu))_{1\leq j\leq \hat{\sigma}}\in\mathbb{C}^{\hat{\sigma}}$.
During the offline stage, $\hat{\sigma}$ interpolation points $\mu_r\in\mathcal{P}_{\rm trial}$ are also selected; these points are
collected in the set $\mathcal{P}_{\rm inter}$.
Notice that $\mathcal{P}_{\rm select}$, the set of parameter values selected by the greedy algorithm of the RB method,
is different from $\mathcal{P}_{\rm inter}$.
During the online stage, the matrix $B^{\hat{\sigma}}\in\mathbb{C}^{\hat{\sigma},\hat{\sigma}}$,
where $B^{\hat{\sigma}}_{ij}=q_i(\mu_j)$, for $1\leq i,j\leq \hat{\sigma}$, is constructed.
Letting $\mu\in\mathcal{P}_{\rm trial}$, we solve for $\lambda^{\hat{\sigma}}(\mu)\in\mathbb{C}^{\hat{\sigma}}$ such that
\begin{equation}
\label{eq:EIMonline}
{B^{\hat{\sigma}}}\lambda^{\hat{\sigma}}(\mu)=q^{\hat{\sigma}}(\mu),
\end{equation}
and compute the rank-$\hat{\sigma}$ interpolation operators defined as follows.
\begin{definition}
Let $1\leq k\leq \hat{\sigma}$. The rank-$k$ interpolation operator $I^{k}$ is defined such that
\label{def_interp}
\begin{equation}
\label{eq:interp}
I^{k}\hat{X}(\mu) := \sum_{r=1}^{k}{\lambda_r^{k}(\mu) \hat{X}(\mu_r)},
\end{equation}
where $\lambda^{k}(\mu)\in\mathbb{C}^{k}$ solves 
\begin{equation}
\label{eq:onlineq}
{B^{k}}\lambda^{k}(\mu)=q^{k}(\mu).
\end{equation}
\end{definition}

Equation \eqref{eq:interp} defines an interpolation in the sense that $I^{k}\hat{X}_{p_r}(\mu)=\hat{X}_{p_r}(\mu)$
for all $1\leq r\leq k$ and all $\mu\in \mathcal{P_{\rm trial}}$.
The formula $\hat{X}_p(\mu)\approx(I^{\hat{\sigma}} \hat{X})_p(\mu)$, for all $\mu\in\mathcal{P_{\rm trial}}$ and all $p\in\{1, ..., \sigma\}$,
provides the approximate interpolation formula searched for in \eqref{eq:interpb}.

\begin{definition}
The residual operator $\delta^{\hat{\sigma}}$ is defined by
\label{def_delta}
\begin{equation}
\label{eq:delta}
\delta^{\hat{\sigma}} := {\rm Id}-I^{\hat{\sigma}}.
\end{equation}
\end{definition}

Algorithm \ref{algo2} presents the construction of the function $q^{\hat{\sigma}}$ by a greedy algorithm during the offline
stage.
This EIM algorithm is a variant from the classical one, described in~\cite{Maday}. The differences stand in the definition of the interpolation operator~\eqref{eq:EIMonline}, the linear system
~\eqref{eq:onlineq} to solve during the online calls, and the definition of the $B^k$ matrix. In particular, the present variant leads to the approximation~\eqref{eq:interp}, which is nonintrusive in the sens
that $I^{k}\hat{X}(\mu)$ is obtained as a linear combination
of evaluations of $\hat{X}$ at some parameter values $\mu_r$. The classical EIM can recover such a property, but to the price of an additional change of basis between $q_k(\cdot)$ and $\hat{X}_{p_k}(\cdot)$.
However, contrary to the classical EIM, the variant needs the additional change of basis to be able to compute an approximation between learning points, namely for $\mu\in\mathcal{P}_{\text{trial}}\backslash
\mathcal{P}$. We refer to \cite[Section~6.8]{Casenave_phd} for more details about the differences between the EIM variant considered here and the classical algorithm.

\begin{algorithm}[h!]
	\caption{Offline stage of the EIM}
	\label{algo2}
	\begin{algorithmic}[1]
        \STATE {Choose $\hat{\sigma}>1$}
        \hfill \COMMENT{Number of interpolation points}
	\STATE {Set $k:=1$}
        \STATE {Compute $\displaystyle p_1:=\underset{p\in \{1,..., \sigma\}}{\textnormal{argmax}}\|\hat{X}_p(\cdot)\|_{\ell^\infty(\mathcal{P}_{\text{trial}})}$}
        \STATE {Compute $\displaystyle \mu_1:=\underset{\mu\in \mathcal{P}_{\text{trial}}}{\textnormal{argmax}}|\hat{X}_{p_1}(\mu)|$ and set
                $\displaystyle \mathcal{P}_{\rm inter}=\{\mu_1\}$}
        \hfill \COMMENT{First interpolation point}
	\STATE {Set $\displaystyle q_1(\cdot):=\frac{\hat{X}_{p_1}(\cdot)}{\hat{X}_{p_1}(\mu_1)}$}
       \hfill \COMMENT{First basis function}
        \STATE {Set $B^1_{11}:=1$}
       \hfill \COMMENT{Initialize $B$ matrix}
        \WHILE {$k<\hat{\sigma}$} 
		\STATE Compute $\displaystyle p_{k+1}:=\underset{p\in \{1,..., \sigma\}}{\textnormal{argmax}}\|(\delta^k \hat{X})_p(\cdot)\|_{\ell^\infty(\mathcal{P}_{\text{trial}})}$
                \STATE Compute $\displaystyle \mu_{k+1}:=\underset{\mu\in \mathcal{P}_{\text{trial}}}{\textnormal{argmax}}|(\delta^k \hat{X})_{p_{k+1}}(\mu)|$
                \hfill \COMMENT{$(k+1)$-th interpolation point}  
                \STATE Set $\displaystyle \mathcal{P}_{\rm inter}:=\mathcal{P}_{\rm inter}\cup\{\mu_{k+1}\}$
                \hfill \COMMENT{Update of $\mathcal{P}_{\rm inter}$}
                \STATE Set $\displaystyle q_{k+1}(\cdot):=\frac{(\delta^k \hat{X})_{p_{k+1}}(\cdot)}{(\delta^k \hat{X})_{p_{k+1}}(\mu_{k+1})}$
                \hfill \COMMENT{$(k+1)$-th basis function}
                \STATE $\displaystyle B^{k+1}_{ij}:=q_i(\mu_j)$, $1\leq i,j\leq {k+1}$
                \hfill \COMMENT{$(k+1)$-th $B$ matrix}
                \STATE $k\leftarrow k+1$
                \hfill \COMMENT{Increment the size of the interpolation}
	\ENDWHILE
\end{algorithmic}
\end{algorithm}

\begin{definition}
\label{def:newprox_formula}
The new formula for computing the error bound is
\begin{equation}
\mathcal{E}_4(\mu):={\tilde{\beta}_{\mu}}^{-1}\left(\sum_{r=1}^{\hat{\sigma}}{\lambda^{\hat{\sigma}}_r(\mu) V_r}\right)^{\frac{1}{2}},
\end{equation}
where $\lambda^{\hat{\sigma}}(\mu)$ is the solution to \eqref{eq:EIMonline}. We recall that
$V_r=\left\|G_{\mu_r}\hat{u}_{\mu_r}\right\|_{\mathcal{V}}^2$. 
\end{definition}

\begin{proposition}
The computation of the formula $\mathcal{E}_4$ is well defined, and this formula is online-efficient.
\end{proposition}
\begin{proof}
Owing to \cite[Theorem 1]{Maday}, the matrix $B$ is upper triangular with diagonal unity. Hence, $\det{B}=1$ and $B$ is
guaranteed to be invertible. The online procedure of EIM, consisting in solving a linear system defined by the matrix $B$,
is thus well defined. Then, since the EIM procedure in carried out on $\hat{X}_p(\mu)$, for all $p\in\{1,...,\sigma\}$
and all $\mu\in\mathcal{P}_{\rm trial}$, all the computations involved are of complexity independent of $N$,
even the offline part of the EIM. Finally, the complexity of the online part of EIM only depends on $\hat{\sigma}$.
\end{proof}

\begin{remark}[Stopping criterion in Algorithm \ref{algo2}]
For ease of presentation, we chose a simple stopping criterion based on an a priori fixed maximum number of interpolation points.
In practice, one possibility is to stop the algorithm when the maximal approximation error in the EIM is below a prescribed
value, by monitoring the quantity $(\delta^k \hat{X})_{p_{k+1}}(\mu_{k+1})$.
\end{remark}

\begin{remark}[Interpolation errors]
\label{rmkerroracc}
As already observed, $\mathcal{E}_4$ does not equal $\mathcal{E}_1$ in exact arithmetics owing to interpolation errors (when
$\hat{\sigma} < \sigma$).
Thus, although Algorithm \ref{algo2} yields an accurate approximation of $\hat{X}_p(\mu)$,
a given interpolation error on $\hat{X}_p(\mu)$ does not directly
translate into a bound on the difference between $\mathcal{E}_1(\mu)$ and $\mathcal{E}_4(\mu)$ 
(the latter depending also on $\delta$, $s$, and $S$, as well as on $\tilde{\beta}_\mu$).
We observe in our numerical experiments that these latter errors are lower than the errors incurred in
the evaluation of $\mathcal{E}_2$ (due to round-off errors)
and in the evaluation of $\mathcal{E}_3$ (due to the poor conditioning of $T$).
\end{remark}

\begin{remark}[Non affine dependence]
When the affine dependence assumption is not available (see Definition~\ref{def:affine_assump}), one can look for an approximation of $a_\mu$ in the following form:
\begin{equation}
\label{eq:variationalcrime}
a_\mu(u,v)\approx\sum_{k=1}^d{\alpha_k(\mu)a_k(u,v)}, \qquad \forall u,v \in\mathcal{V}.
\end{equation}
In the reduced basis context, this approximation is usually computed using the EIM. We saw that the formula~\eqref{eq:E2} for $\mathcal{E}_2$ makes use of this affine decomposition to ensure online efficiency, and
therefore does not account for the approximation in the operator.
On the contrary, the formulae~\eqref{eq:E1} for $\mathcal{E}_1$ and~\eqref{eq:Est3} for $\mathcal{E}_3$ use the exact operator.
% An important property of the new formula $\mathcal{E}_3$ in Definition~\ref{def:newprox_formula} is that since the $V_r=\left\|G_{\mu_r}\hat{u}_{\mu_r}\right\|_{\mathcal{V}}^2$ are computed
% on the complete residual (and not the one using the affine decomposition), our new procedure accounts for the variational crime as well.
\end{remark}

\subsection{Illustration}
\label{sec:numerics}
Consider as in \cite{casenave} a one-dimensional linear diffusion problem, namely the boundary value problem
$-u''+\mu u=1$ on $]0,1[$ with $u(0)=u(1)=0$, with parameter $\mu\in\mathcal{P}:=[1,100]$. The analytic solution is
\begin{equation}
u(x)=-\frac{1}{\mu}\left(\cosh\left(\sqrt{\mu}x\right)-1\right)+
\frac{\cosh\left(\sqrt{\mu}\right)-1}{\mu\sinh\left(\sqrt{\mu}\right)}\sinh\left(\sqrt{\mu}x\right).
\end{equation}
The Lax--Milgram theory is valid, and the coercivity constant is bounded from below by $1$ in the $H^1$-norm.
The error bound is given by $\mathcal{E}_1(\mu)=\|G_\mu \hat{u}_\mu \|_{H^1\left(]0,1[\right)}$.
Lagrange $\mathbb{P}_1$ finite elements are used with uniform mesh cells of length 0.005.
The set $\mathcal{P}_{\rm trial}$ consists of $1000$ points uniformly distributed in $\mathcal{P}$.
The RB method is carried out until the formula $\mathcal{E}_2$ suffers from round-off errors,
which already happens for a reduced basis of size $\hat{N}=7$ (since $d=2$, we obtain $\sigma=225$). A direct solver is
used, so that the only adverse phenomenon to compute the error bound are round-off errors.

In Figure \ref{fig:plot}, we see that the classical formula $\mathcal{E}_2$ is not valid for computing the error bound
with any tolerance below $10^{-7}$, whereas the formulae $\mathcal{E}_1$, $\mathcal{E}_3$ and
$\mathcal{E}_4$ are valid with tolerances down to $10^{-14}$. The difference is of $7$ orders of magnitude ; given that
$\sqrt{\epsilon}\approx 10^{-7}$, this is consistent with Remark \ref{direct_case} and Section \ref{sec:first_rewriting}.

\begin{figure}[ht]
\centerline{
\includegraphics [width=9cm] {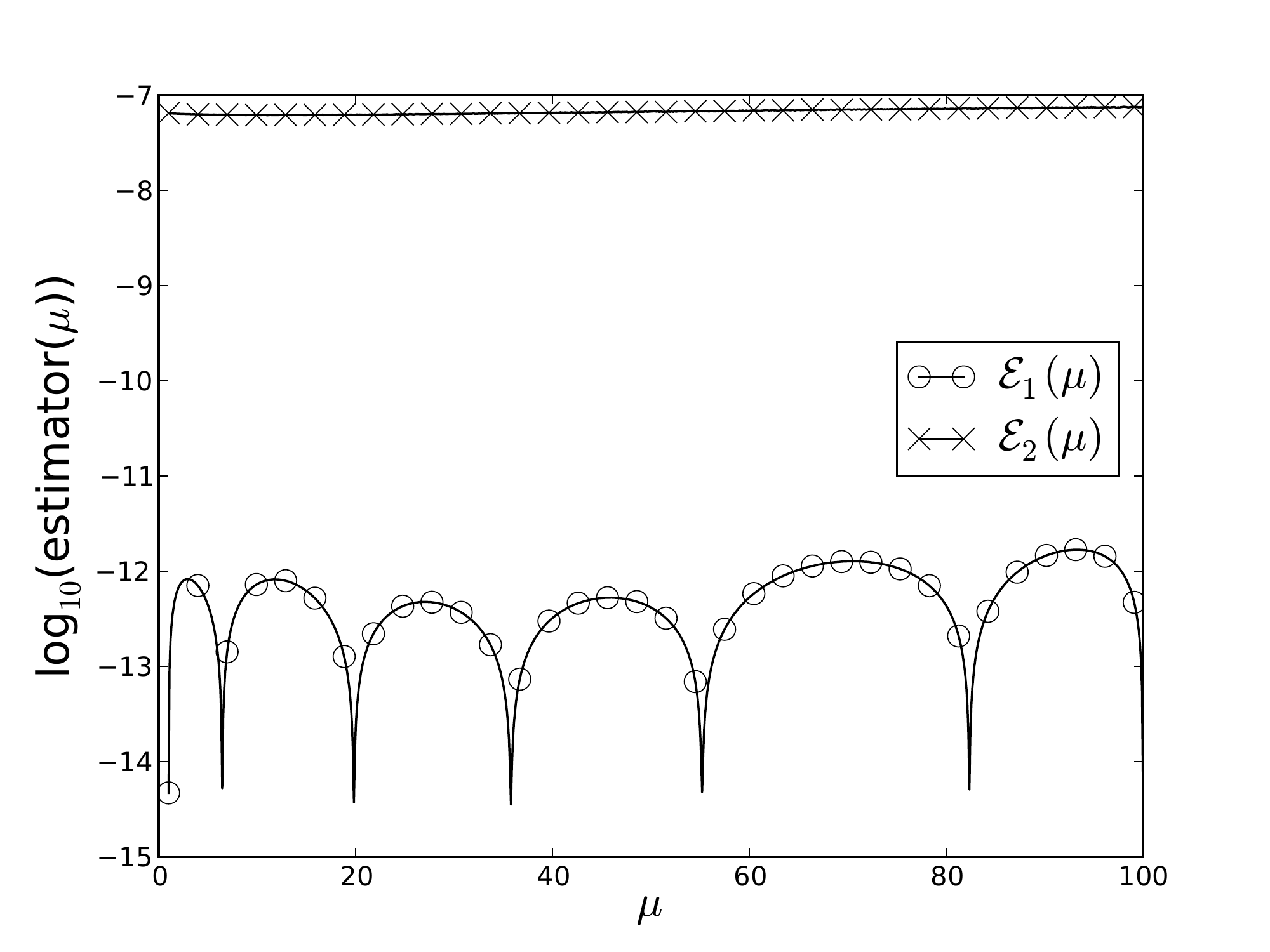}
\includegraphics [width=9cm] {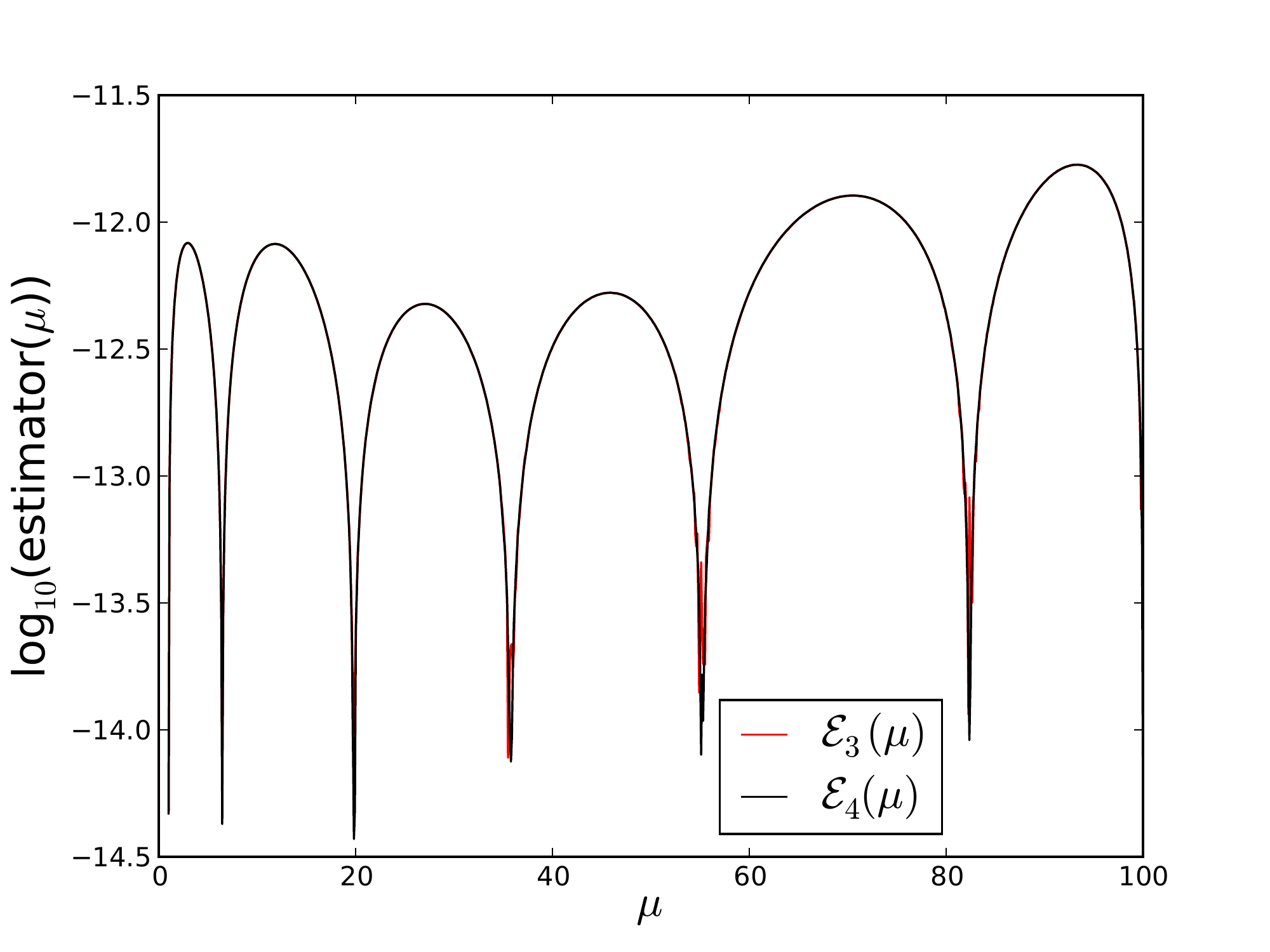}}
\caption{Error bound curves with respect to the parameter. The formula $\mathcal{E}_4$ is computed with $\hat{\sigma}=23$.}
\label{fig:plot}
\end{figure}

In Figure \ref{fig:plotbis}, we observe that instabilities occur in the formula $\mathcal{E}_3$, especially for parameter values close to the
elements of $\mathcal{P}_{\rm select}$. This is due to the poor conditioning of the matrix $T$ when solving \eqref{eq:sysT}.
The new formula $\mathcal{E}_4$ based on the EIM is seen to introduce much less numerical errors than $\mathcal{E}_3$.

\begin{figure}[ht]
\centerline{
\includegraphics [width=9cm] {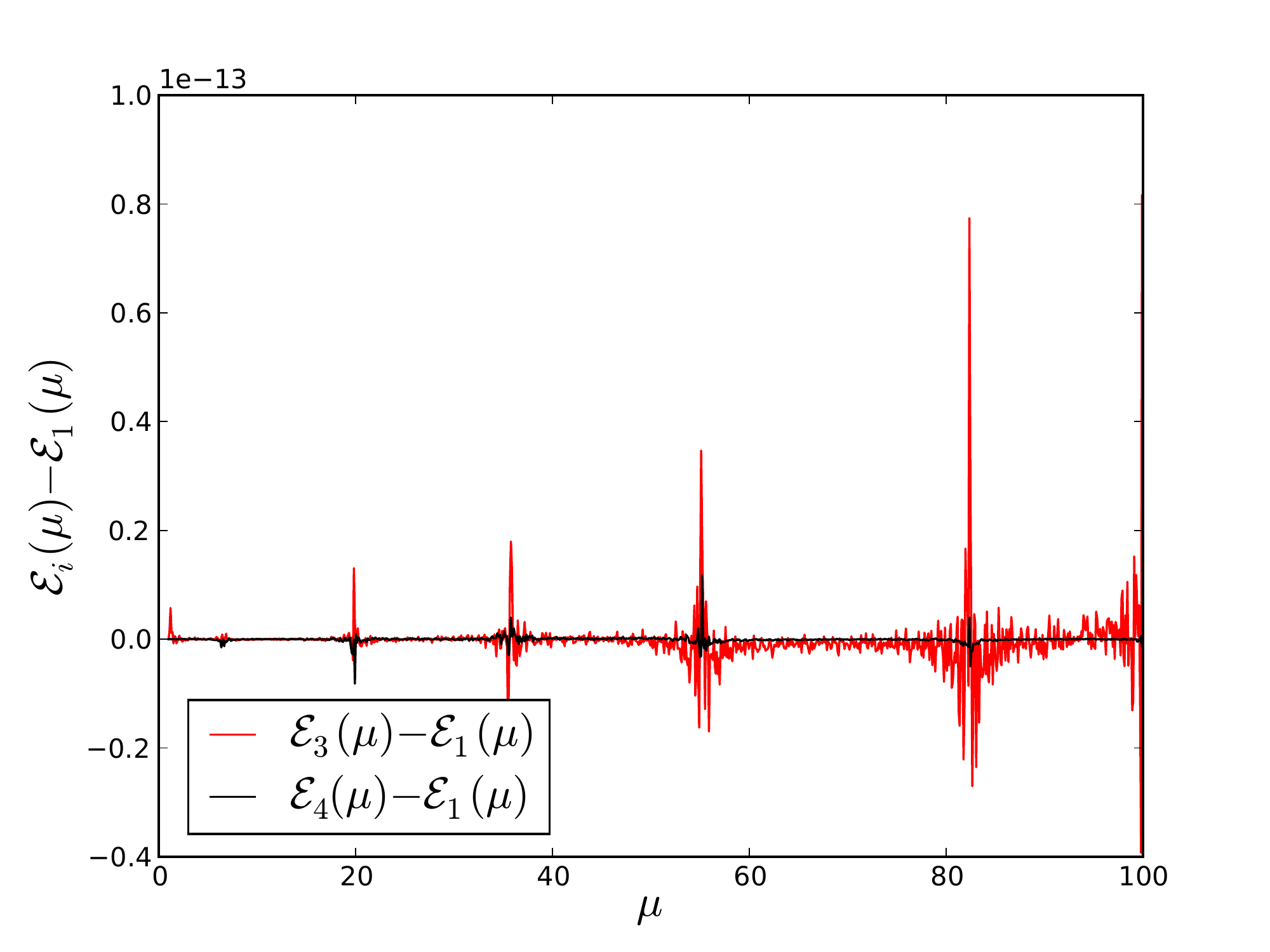}}
\caption{Comparison of the formulae $\mathcal{E}_3$ and $\mathcal{E}_4$, with respect to the formula $\mathcal{E}_1$.}
\label{fig:plotbis}
\end{figure}

% % In Figure \ref{fig:plotter}, we plot $\underset{\mu\in\mathcal{P}_{\rm select}}{\max}(\mathcal{E}_4(\mu))$ 
% % as a function of $\hat{\sigma}$. From this figure and Definition \ref{def_val}, we deduce that
% % for $\hat{\sigma}< 23$, the formula $\mathcal{E}_4$ is valid for any tolerance larger than $10^{-12}$.
% % If we want to consider a tolerance of the order of $10^{-14}$, we need $\hat{\sigma}\geq 23$.
% % 
% % \begin{figure}[ht]
% % \centerline{
% % \includegraphics [width=9cm] {images/max_E4.pdf}}
% % \caption{$\underset{\mu\in\mathcal{P}_{\rm select}}{\max}(\mathcal{E}_4(\mu))$ as a function of $\hat{\sigma}$.}
% % \label{fig:plotter}
% % \end{figure}

\subsection{Procedure 3: improvement of Procedure 2 using a stabilized EIM}
\label{sec:stabEIM}
In practice, round-off errors are accumulated during the loop in Algorithm \ref{algo2}, and if we keep increasing the number
of interpolation points, the coefficients of the matrix $B$ suffer from round-off errors, so that the relation
$\textnormal{det}(B)=1$ no longer holds. The matrix $B$ becomes non invertible at some stage.
To solve this problem, we now propose a numerical stabilization of EIM based on the following property:
\begin{property}
\label{prop_stab}
There holds
\begin{equation}
\label{eq:prop_stab}
\forall i<j,~I^j\circ I^i = I^i,
\end{equation}
where the interpolation operators $I^j$ are defined by \eqref{eq:interp}.
\end{property}
\begin{proof}
Using \cite[Lemma 1]{Maday}, $I^i \hat{X}\in\text{Span}\left(q_1,...,q_i\right)$
and $I^i v=v$ for all $v\in\text{Span}\left(q_1,...,q_i\right)$. Therefore, $I^j\circ I^i \hat{X} = I^i \hat{X}$ for all $i<j$.
\end{proof}
In our numerical experiments, we observe that, as the number of iterations of the greedy procedure for the EIM grows, the relation \eqref{eq:prop_stab}
is no longer verified numerically, due to accumulation of round-off errors. 
These numerical instabilities can be compensated in the same fashion as the
Gram--Schmidt orthonormalization procedure is stabilized (see \cite[chapter 5.2.8]{golub}).
The Gram--Schmidt algorithm transforms a linearly
independent family of vectors $\{v_i\}$ into an orthonormal basis $\{u_i\}$. To simplify the presentation, we suppose in
what follows that the normalization step is not carried out. Consider
the orthogonalization step for the $k$-th vector. We denote by $\Pi^k$ the projection operator on $\text{Span}(u_1,...,u_k)$,
and $\delta^k:={\rm Id}-\Pi^k$. For the EIM, we suppose that $(k-1)$ interpolation operators $I^i$, $1\leq i \leq k-1$, have
been constructed, and we wish to construct the $k$-th interpolation operator $I^k$.
A comparison between the stabilized Gram--Schmidt orthonormalization procedure and
the proposed stabilization for the EIM is presented in Table \ref{tab1}.

\renewcommand{\arraystretch}{1.5}
\begin{table}
\begin{tabular}{|c |c |c|}
 \cline{2-3}
\multicolumn{1}{c|}{} & stabilized Gram--Schmidt & stabilized EIM\\
 \hline
\small global input &\small  $(v_1,...,v_{\hat{\sigma}})$ basis of $\mathbb{C}^{\hat{\sigma}}$ & \small $v:\mathcal{P}_{\rm trial}\rightarrow\mathbb{C}^{\hat{\sigma}}$\\
 \hline
\small classical residual step $k$ & \small $\delta^k v_k=v_k-\Pi^{k} v_k$ &\small  $(\delta^k v)(\mu)=v(\mu)-(I^{k}v)(\mu)$\\
 \hline
\multirow{4}*{\small intermediate residuals step $k$} & \small $\delta^{k,1}_{\rm stab}v_k=v_k-\Pi^{1} v_k$ & \small $(\delta^{k,1}_{\rm stab}v)(\mu)=v(\mu)-(I^{1}v)(\mu)$\\[3pt]
& \small $\delta^{k,2}_{\rm stab}v_k=\delta^{k,1}_{\rm stab}v_k-\Pi^{2} \delta^{k,1}_{\rm stab}v_k$, &  \small $(\delta^{k,2}_{\rm stab}v)(\mu)=(\delta^{k,1}_{\rm stab}v)(\mu)-I^{2} (\delta^{k,1}_{\rm stab}v)(\mu)$,\\[3pt]
& \small $\vdots$ &\small  $\vdots$ \\[3pt]
&\small  $\delta^{k,k}_{\rm stab}v_k=\delta^{k,k-1}_{\rm stab}v_k-\Pi^{k} \delta^{k,k-1}_{\rm stab}v_k$ &\small  $(\delta^{k,k}_{\rm stab}v)(\mu)=(\delta^{k,k-1}_{\rm stab}v)(\mu)-I^{k}( \delta^{k,k-1}_{\rm stab}v)(\mu)$ \\[3pt]
\hline
\small  stabilized residual step $k$ & \small $\delta^{k}_{\rm stab}v_k=\delta^{k,k}_{\rm stab}v_k$ & \small $(\delta^{k}_{\rm stab}v)(\mu)=(\delta^{k,k}_{\rm stab}v)(\mu)$ \\[3pt]
 \hline
\multirow{2}*{\small global output} & \small $(\delta^{1}_{\rm stab}v_1,\delta^{2}_{\rm stab}v_2,...,\delta^{\hat{\sigma}}_{\rm stab}v_{\hat{\sigma}})$ & \multirow{2}*{\small $(I^{\hat{\sigma}}v)(\mu)$}\\
&{\small orthogonal basis of $\text{Span}(v_1,...,v_{\hat{\sigma}})$}&\\
 \hline
\end{tabular}
\vspace{0.2cm}
\caption{\label{tab1} Comparison between stabilized Gram--Schmidt and stabilized EIM.}
\end{table}

\begin{proposition}
Let $k\in\mathbb{N}^*$. In exact arithmetic, the following relations hold for the residuals defined in Table \ref{tab1}:
$\delta^k_{\rm stab} v=\delta^k v$.
\end{proposition}
\begin{proof}
We prove by recursion that, for all $i\leq k$, $\delta_{\rm stab}^{k,i}=\delta^i$.
The case $i=1$ is clear from the definition of the first intermediate residual in Table \ref{tab1}.
Let $i\leq k$ and suppose that $\delta_{\rm stab}^{k,i-1}={\rm Id}-I^{i-1}$ for the EIM. There holds
\begin{equation}
\delta^{k,i}_{\rm stab}=\delta_{\rm stab}^{k,i-1}-I^i\circ\delta_{\rm stab}^{k,i-1}={\rm Id}-I^{i-1}-I^i+ I^{i}\circ I^{i-1}={\rm Id}-I^i=\delta^i,
\end{equation}
since $I^{i}\circ I^{i-1}=I^{i-1}$ owing to Property \ref{prop_stab}. The results follow from the case $i=k$.
The same relation is proved likewise for the Gram--Schmidt procedure, for which $\Pi^{i}\circ\Pi^{i-1}=\Pi^{i-1}$ holds as well.
\end{proof}

\begin{definition}[Stabilized EIM]
The stabilized EIM consists in the same offline procedure as the one described in Section \ref{sec:principle},
except that the residuals $\delta^k$ are replaced by the stabilized residuals $\delta^k_{\rm stab}$
defined in Table \ref{tab1}. The online stage is the same as that of the classical EIM. 
\end{definition}

The stabilized Gram--Schmidt procedure generates a set of vectors much less polluted by round-off errors
(see \cite{bjorck,giraud}).
By analogy we expect that the stabilized EIM produces a more accurate interpolation procedure than the classical EIM,
that is, much less polluted by round-off errors. This is numerically verified in Figure~\ref{fig:detcond}, where
$\text{det}(B^{\hat{\sigma}})$ and $\text{cond}(B^{\hat{\sigma}})$ are represented as a function of $\hat{\sigma}$.
We consider the test case described in Section \ref{sec:numerics}, where we recall that $\hat{N}=7$, $d=2$, and $\sigma=225$.
If the method is stable, then $\text{det}(B^{\hat{\sigma}})=1$ should hold throughout the process.
Figure~\ref{fig:detcond} shows that the stabilized EIM behaves as intended.
The classical EIM curve stops since the matrix $B^{\hat{\sigma}}$ becomes noninvertible at some point:
a parameter already in $\mathcal{P}_{\rm inter}$ has been selected by the greedy algorithm.
Invertibility can be recovered artificially by ensuring that the new interpolation point is not an element
of the current set $\mathcal{P}_{\rm inter}$. We call this procedure EIM with unique choice.
However, this fix is not completely satisfactory, since $\text{det}(B^{\hat{\sigma}})=1$ is not satisfied.
Moreover, $\text{cond}(B^{\hat{\sigma}})$ is much more ill-behaved with this procedure than with the stabilized EIM.

\begin{figure}[ht]
\centerline{
\includegraphics [width=9cm] {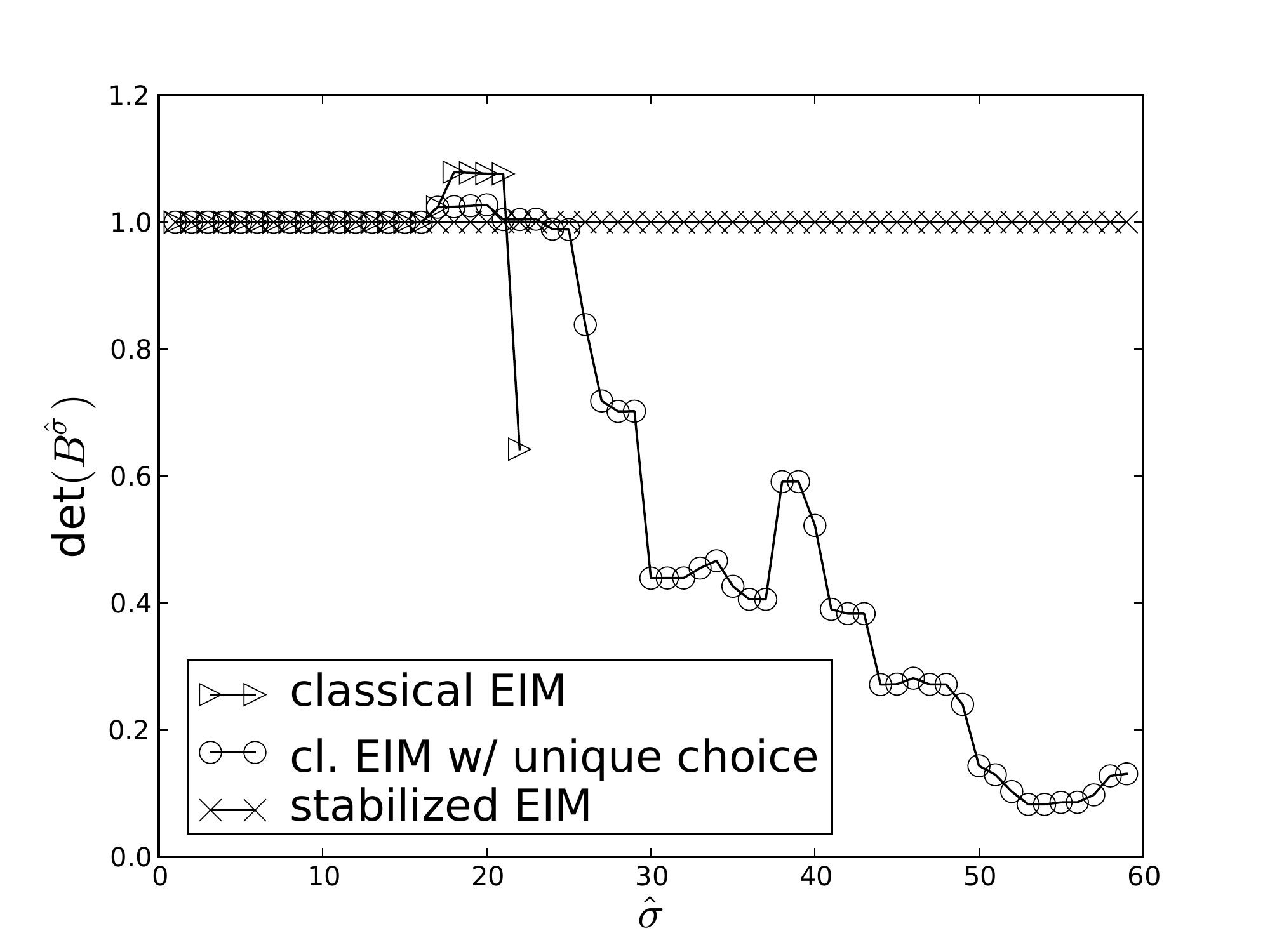}
\includegraphics [width=9cm] {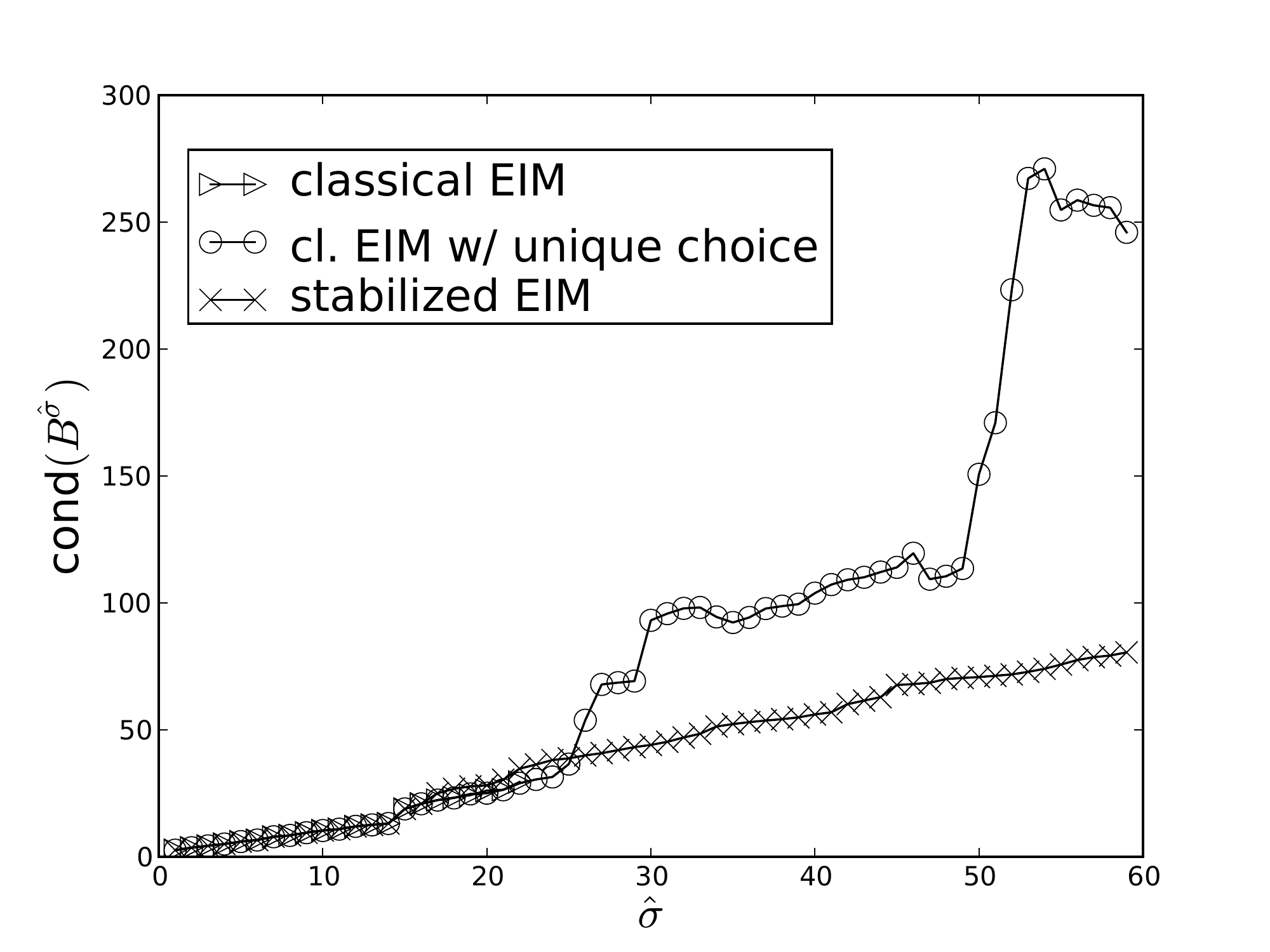}}
\caption{Determinant (left) and condition number (right) of the matrix $B^{\hat{\sigma}}$ as a function of
$\hat{\sigma}$, for the classical EIM, the classical EIM with unique choice, and the stabilized EIM.
The classical EIM curves stop at $21$ interpolation points since $B^{\hat{\sigma}}$ becomes non invertible at $22$ points.}
\label{fig:detcond}
\end{figure}

\begin{remark}[Computational cost and variant of stabilized EIM]
The computational cost of the stabilized EIM is more than that of the classical EIM, since the stabilized residual
requires as many calls to a classical residual as the number of selected interpolation points (i.e. the scaling with $\hat{\sigma}$
is $\hat{\sigma}^2$ for the stabilized EIM as opposed to $\hat{\sigma}$ for the classical EIM). 
One can think of a cheaper procedure by monitoring $\textnormal{det}(B^{\hat{\sigma}})$ and adding some intermediate residuals
$\delta^{k,j}_{\rm stab}$ until $\textnormal{det}(B^{\hat{\sigma}})$ is close enough to $1$.
\end{remark}

\subsection{Summary}

The advantages and drawbacks of the four considered formulae for computing the error bound are summarized in Table \ref{tab3}.
To estimate the computational complexity of the methods, we keep only the leading order in operation count.
We denote the complexity of the resolution
of \eqref{eq:compute_G} by $N_{\rm sol}$. The linear systems of size $\sigma$, $\hat\sigma$, and $\hat{N}$ are supposed to be solved
by a direct solver, hence with complexity proportional to $\sigma^3$, $\hat\sigma^3$, and $\hat{N}^3$, respectively.
For the offline stage of $\mathcal{E}_2$ and
$\mathcal{E}_3$, we have to evaluate respectively $d\hat{N}+1$ and $\sigma$ times the functional $G_\mu$, which requires to
solve \eqref{eq:compute_G}. For the offline stage of $\mathcal{E}_4$, let $M$ denote the cardinality of
$\mathcal{P}_{\rm trial}$. The $k$-loop in Algorithm \ref{algo2} requires at each step to compute a maximum over $\sigma$ different
$\ell^\infty(\mathcal{P}_{\rm trial})$ norms, and then to solve a linear system of size $k$, leading to a complexity
of $\hat{\sigma}^4 \sigma M+\hat\sigma N_{\rm sol}$. If the stabilized EIM is used instead for $\mathcal{E}_4$, each residual evaluation
in the $k$-loop
requires solving $k$ linear systems of size $1$ to $k$, leading to a complexity of $\hat{\sigma}^5 \sigma M+\hat\sigma N_{\rm sol}$.
% Then, the precomputations of error bound values at the $\hat{\sigma}$ interpolation points is of complexity $\hat{\sigma}N_{\rm sol}$.
For the online stage, all the formulae require to solve the problem $\hat{E}_\mu$ of size $\hat{N}$. Moreover,
$\mathcal{E}_2$ additionally requires a linear combination of size $\sigma$, whereas
$\mathcal{E}_3$ and $\mathcal{E}_4$ require to solve a linear system of size $\sigma$ and $\hat{\sigma}$ respectively.
We notice that if $N_{\rm sol} \gg \hat{\sigma}^4\sigma M$ and $\hat{\sigma} < d\hat{N}+1$, then the offline stage of $\mathcal{E}_4$
with stabilized EIM requires less precomputations than the offline stage of $\mathcal{E}_2$.

\renewcommand{\arraystretch}{1.5}
\begin{table}
\begin{tabular}{|c |c |c |c |c|}
 \hline
Property & $\mathcal{E}_1$ & $\mathcal{E}_2$ & $\mathcal{E}_3$ & $\mathcal{E}_4$\\
 \hline
\small Online efficient & No & Yes & Yes & Yes \\
 \hline
\small Unconditionally well-posed & Yes & Yes & No & Yes \\
 \hline
\small $\epsilon$-dependence of the accuracy & $\epsilon$ & $\sqrt{\epsilon}$ & $\epsilon$, if well-posed & $\epsilon$ \\
 \hline
\multirow{2}*{\small Equals $\mathcal{E}_1$ in exact arithmetics} & \multirow{2}*{--} & \multirow{2}*{Yes} & \multirow{2}*{Yes} & Yes, if $\hat{\sigma}=\sigma$\\[-3pt]
& & & & No, if $\hat{\sigma}<\sigma$\\
 \hline
\multirow{2}*{\small Complexity of the offline stage} & \multirow{2}*{--} & \multirow{2}*{$(d\hat{N}+1) N_{\rm sol}$} & \multirow{2}*{$\sigma N_{\rm sol}$} & $\hat{\sigma}^4 \sigma M + \hat{\sigma}N_{\rm sol}$ {\small with classical EIM }\\[-3pt]
& & & & $\hat{\sigma}^5 \sigma M + \hat{\sigma}N_{\rm sol}$ {\small with stabilized EIM}\\
 \hline
\small Complexity of the online stage & -- & $\hat{N}^3+\sigma$ & $\hat{N}^3+\sigma^3$ & $\hat{N}^3+\hat{\sigma}^3$ \\
 \hline
\end{tabular}
\vspace{0.2cm}
\caption{\label{tab3} Comparison of the considered formulae for computing the error bound.}
\end{table}

\section{Application to a three-dimensional acoustic scattering problem}
\label{sec:acoustics}

\subsection{Formulation of the problem}

We consider a ball $\Omega^i\subset\mathbb{R}^3$ with boundary $\Gamma$ and $\Omega^e:=\mathbb{R}^3\backslash\overline{\Omega^i}$,
see Figure \ref{fig:geo}. We consider a monopole source located in $\Omega^e$. The surface of the ball is impedant,
meaning that any incident wave will be partially absorbed and partially scattered. The proportion of absorbed and scattered
parts is quantified by the impedance coefficient $\mu$, which is used in a Robin boundary condition at $\Gamma$.
We are interested in the computation of the scattered field $p_{\rm sc}$ in $\Omega^e$. We denote
$p_{\rm inc}$ the known pressure field created by the source in the absence of the sphere;
the total acoustic field in $\Omega^e$ is the sum of $p_{\rm inc}$ and $p_{\rm sc}$.

\begin{figure}[ht]
	\centering
	\includegraphics [width=0.65\textwidth] {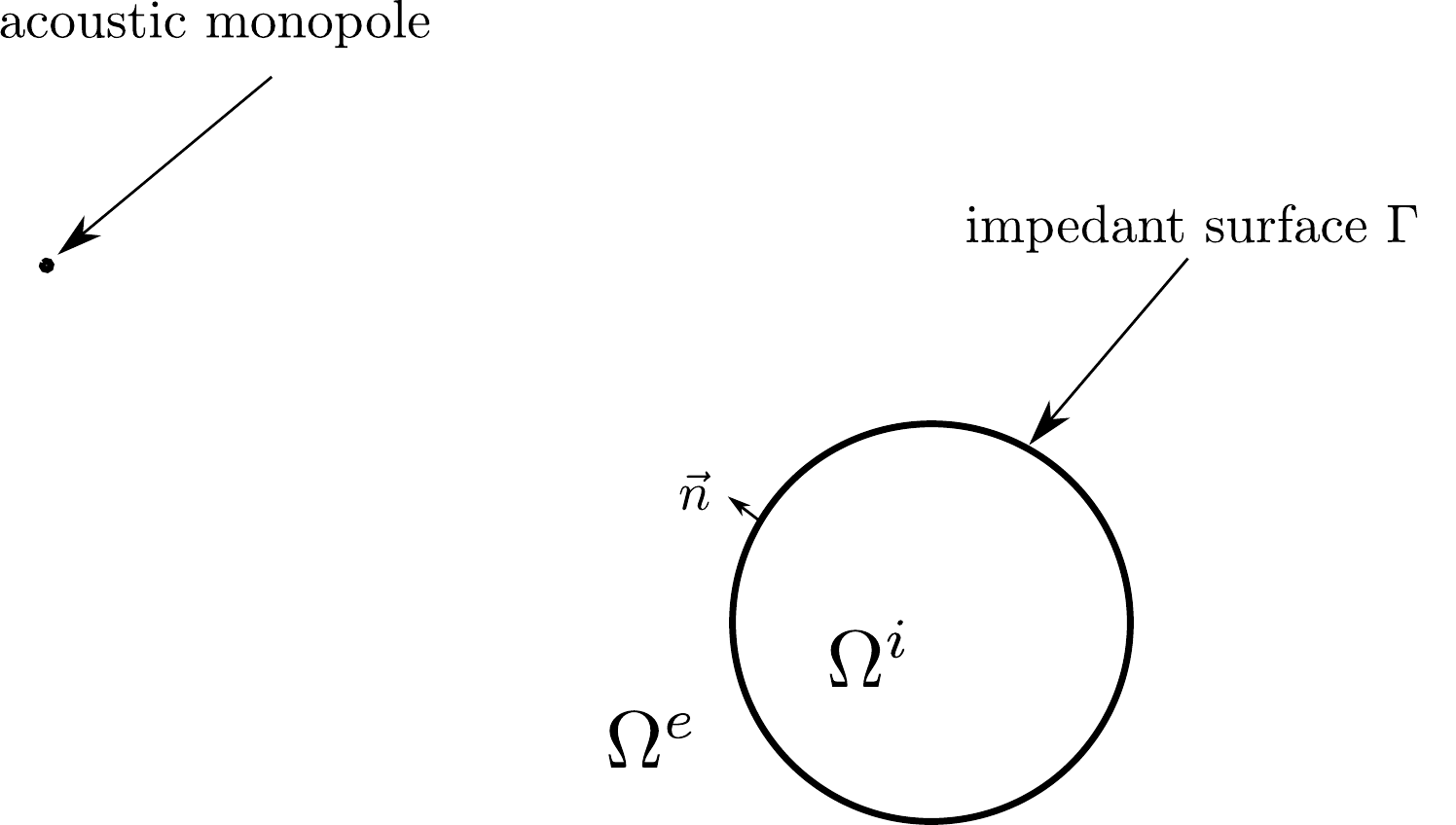}
	\caption{Geometry for the three-dimensional acoustic scattering problem}
\label{fig:geo}
\end{figure}

We define the distribution $v:\Omega^e\cup\Omega^i\longrightarrow\mathbb{C}$ such that $v_{|\Omega^i}=-p_{\rm inc}$,
$v_{|\Omega^e}=p_{\rm sc}$. We denote $\lambda$ and $\chi$ the jumps
of the Neumann and Dirichlet traces of $v$ across $\Gamma$. The Robin boundary condition writes
$\lambda+\frac{ik}{\mu}\chi=0$. Since $v$ solves the homogeneous Helmholtz equation in $\Omega^e$
and in $\Omega^i$ and satisfies the Sommerfeld radiation condition at infinity, there holds
\begin{equation}
\label{eq:form_rep}
v=-\mathcal{S}\lambda + \mathcal{D}\chi\textnormal{~~~in }\Omega^{e}\cup\Omega^{i},
\end{equation}
where $\mathcal{S}$ and $\mathcal{D}$ are respectively the single- and double-layer potentials.
Taking the interior Dirichlet and Neumann traces of $v$ in equation \eqref{eq:form_rep} and injecting the Robin boundary condition,
we obtain
\begin{equation}
\label{eq:acous1}
\left[
\begin{array}{cc}
 N-\frac{ik}{2\mu}I & \tilde{D}\\
 D & -S-\frac{i\mu}{2k}I
\end{array}
\right]\left[
\begin{array}{cc}
\chi\\
\lambda
\end{array}\right]=\left[
\begin{array}{cc}
\gamma_1^{-} p_{\text{inc}}\\
-\gamma_0^{-} p_{\text{inc}}
\end{array}
\right],
\end{equation}
where $k$ is the wave number of the monopole source, $N$, $\tilde{D}$, $D$ and $S$ are classical boundary integral operators (see \cite{sauter}),
and $\gamma_0^{-} p_{\text{inc}}$ and $\gamma_1^{-} p_{\text{inc}}$ are respectively the interior Dirichlet and
Neumann traces of the known function $p_{\text{inc}}$.
Solving one of these two equations, together with the Robin boundary condition, is sufficient. The software we are using,
ACTIPOLE (see \cite{actipole2,actipole1}), deals with the block system defined in \eqref{eq:acous1}, which presents the
advantage of being invertible for all frequencies of the source when the surface  $\Gamma$ is Lipschitz.
We denote $A_\mu$ the block operator defined
by the left-hand side of \eqref{eq:acous1}. From \cite{Wendland, mclean, sauter}, we infer that $A_\mu$ is a bounded bijective operator from $H^{\frac{1}{2}}(\Gamma)\times
L^2(\Gamma)$ into $H^{-\frac{1}{2}}(\Gamma)\times L^2(\Gamma)$ (see also \cite{Casenave_phd}). The variational form is as follows:
find $\left(\chi,\lambda\right)\in H^\frac{1}{2}(\Gamma)\times L^2(\Gamma)$
such that for all $(\hat{\chi},\hat{\lambda})\in H^\frac{1}{2}(\Gamma)\times L^2(\Gamma)$,
\begin{equation}
\label{eq:varf}
 \left\{
\begin{aligned}
\left(N\chi-\frac{ik}{2\mu}\chi, \hat{\chi}\right) + \left(\tilde{D}\lambda, \hat{\chi}\right)&= \left(\gamma_1 p_{\text{inc}}, \hat{\chi}\right),\\
\left<\hat{\lambda},D \chi\right> - \left<\hat{\lambda},S\lambda +\frac{i\mu}{2k}\lambda\right> &= -\left<\hat{\lambda},\gamma_0 p_{\text{inc}}\right>,
\end{aligned}
\right.
\end{equation}
where $(\cdot,\cdot)$ denotes the $H^\frac{1}{2}(\Gamma)\times H^{-\frac{1}{2}}(\Gamma)$ duality product and
$<\cdot,\cdot>$ denotes the $L^2(\Gamma)$ inner product.

Let $\mathcal{M}$ be a shape-regular triangular mesh of $\Gamma$ with meshsize $h$, and
let $V^1_h$ and $V^0_h$ be respectively the spaces spanned by continuous piecewise affine polynomials on $\mathcal{M}$ and
piecewise constant polynomials on $\mathcal{M}$.
Let $(\phi_i)_{1\leq i\leq P}$ and $(\psi_j)_{1\leq j\leq P'}$ be the usual bases of $V^1_h$ and $V^0_h$ of size $P$ and $P'$, respectively.
The product space $V^1_h\times V^0_h$ is a conforming approximation of $H^\frac{1}{2}(\Gamma)\times L^2(\Gamma)$.
The discrete problem is derived from a Galerkin procedure on $V^1_h\times V^0_h$ using the boundary element
method (BEM). From \cite{Wendland}, the obtained discrete approximation of the problem \eqref{eq:varf} is inf-sup stable for
$h$ small enough (see also \cite{Casenave_phd}). A direct solver is used, in double-precision format.

\subsection{Application of the RB method}

The RB method has recently been applied to problems solved by means of integral equations in electromagnetism,
see \cite{Stamm, Brown}. In these works, the classical a posteriori error bounds were used.
We are here interested in the application of our improved a posteriori error bounds to such problems.
We take as parameter for the RB method the value of the impedance $\mu$, which is supposed here to be a positive real number.
To recover an affine dependence on the parameter $\mu$, we write the BEM
matrix in the form $A_\mu=a_1(\mu)A_1+a_2(\mu)A_2+a_3(\mu)A_3$, so that $d=3$ in the affine decomposition \eqref{eq:affine_assump}
with $a_1(\mu)=1$, $a_2(\mu)=\frac{1}{\mu}$ and $a_3(\mu)={\mu}$. Specifically,
\begin{equation}
A_1=\left[
\begin{array}{c|c}
\left(N \phi_i,\phi_j\right)_{\indices{i}{P}{j}{P}} & 
\left(\tilde{D} \psi_j,\phi_i\right)_{\indices{i}{P}{j}{P'}} \\
\hline
\left<D \phi_j,\psi_i\right>_{\indices{i}{P'}{j}{P}} &
\left<-S \psi_i,\psi_j\right>_{\indices{i}{P'}{j}{P'}}
\end{array}
\right],
\end{equation}
\begin{equation}
A_2=\left[
\begin{array}{c|c}
-\frac{ik}{2}\left(\phi_i,\phi_j\right)_{\indices{i}{P}{j}{P}} & 
\left(0\right)_{\indices{i}{P}{j}{P'}} \\
\hline
\left(0\right)_{\indices{i}{P'}{j}{P}} &
\left(0\right)_{\indices{i}{P'}{j}{P'}}
\end{array}
\right],\quad
A_3=\left[
\begin{array}{c|c}
\left(0\right)_{\indices{i}{P}{j}{P}} & 
\left(0\right)_{\indices{i}{P}{j}{P'}} \\
\hline
\left(0\right)_{\indices{i}{P'}{j}{P'}} &
-\frac{i}{2k}\left<\psi_i,\psi_j\right>_{\indices{i}{P'}{j}{P'}}
\end{array}
\right].
\end{equation}

In the general-purpose RB, the quantity of interest is the pair of potentials $(\chi,\lambda)$ on $\Gamma$.
For the goal-oriented case, we consider the value of the pressure at a given point in $\Omega^e$. If this point is far
enough from $\Gamma$, approximations can be made in the representation formula for the pressure. This is the far-field
approximation, which consists in a linear form $Q$ acting on the solution pair $(\chi,\lambda)$ as
\begin{equation}
Q(\chi,\lambda)=
\left(
\begin{aligned}
-&ik\frac{e^{-ik\|x\|_2}}{4\pi\|x\|_2}\left(e^{-iky\cdot\frac{x}{\|x\|_2}}\frac{x}{\|x\|_2}
\cdot n(y),\chi(y)\right)\\
&ik\frac{e^{-ik\|x\|_2}}{4\pi\|x\|_2}\int_{\Gamma}\left(e^{-iky\cdot\frac{x}{\|x\|_2}},\lambda(y)\right)
\end{aligned}
\right)\in\mathbb{C}^2.
\end{equation}
For simplicity, we take the Euclidian norm of
vectors in $\mathbb{C}^{P+P'}$ instead of the $H^{\frac{1}{2}}(\Gamma)\times L^2(\Gamma)$ norms of the reconstructed functions.
This way, the Riesz isomorphism $J$ is simply the identity. Therefore, the computation of the terms
$G_{\mu} u_{\mu}$, as well as that of the terms $G_k u_i$, does not require to invert the stiffness matrix as in \eqref{eq:compute_G}.
The Successive Constraint Method is used to compute a lower bound of the inf-sup constant, which
is around $10^{-6}$ in the present examples.

We define two test cases:
(i) one impedant sphere ($d=3$), with $N=584$ and $\mu\in\mathcal{P}:=[0.9,1.1]$,
(ii) two impedant spheres ($d=5$), with $N=1561$ and $\mu\in\mathcal{P}:=[0.99,1.01]^2$.
We present visualizations of the scattered pressure field, at a random value of the parameter $\mu$, for test case
(i) with $\#\mathcal{P}_{\textnormal{trial}}=100$ and $\hat{N}=10$ in Figure \ref{fig:comp1sph} and for test case
(ii) with $\#\mathcal{P}_{\textnormal{trial}}=225$ and $\hat{N}=10$ in Figure \ref{fig:comp2sph}.

\begin{figure}[h!]
\centering
	\includegraphics[width=0.48\textwidth]{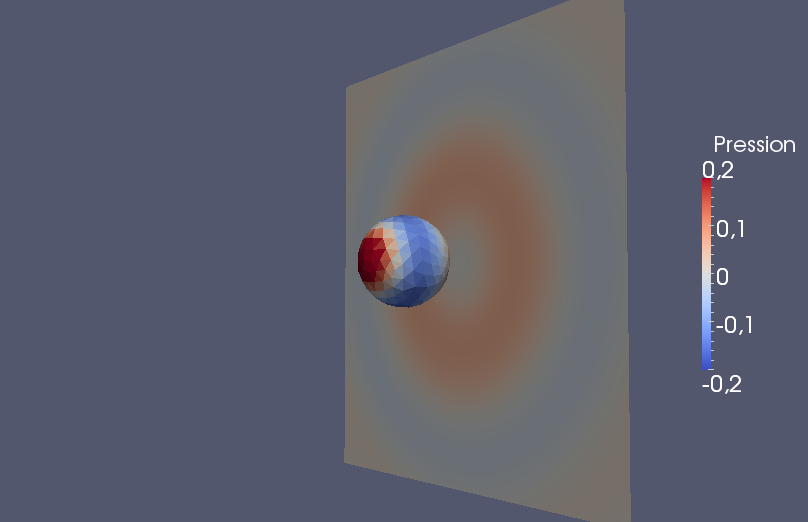}~
	\includegraphics[width=0.48\textwidth]{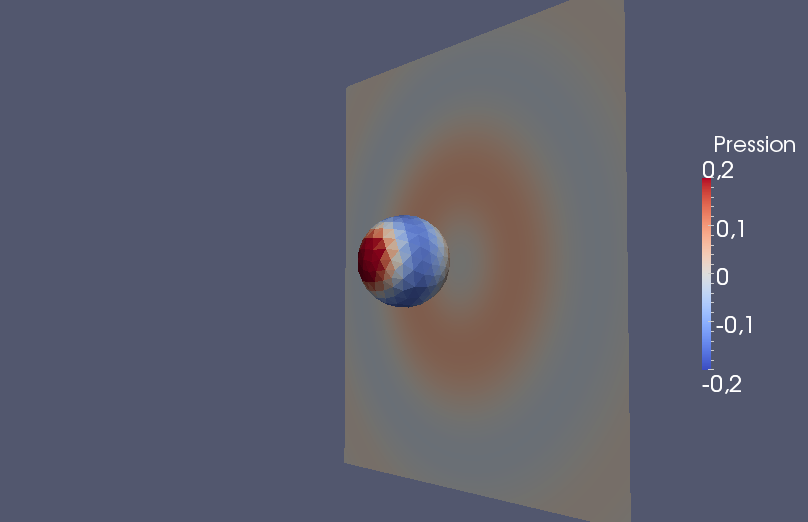}
	\caption{Real part of the pressure field for the BEM solution (left) and the RB solution (right), with a
basis of size 10. The difference between the two fields is less than $10^{-15}$ in infinity norm.}
\label{fig:comp1sph}
\end{figure}

\begin{figure}[h!]
\centering
	\includegraphics[width=0.48\textwidth]{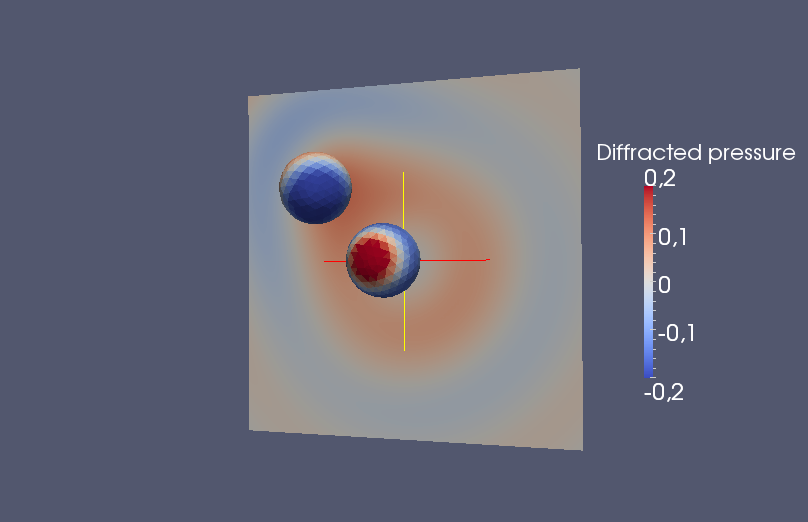}~
	\includegraphics[width=0.48\textwidth]{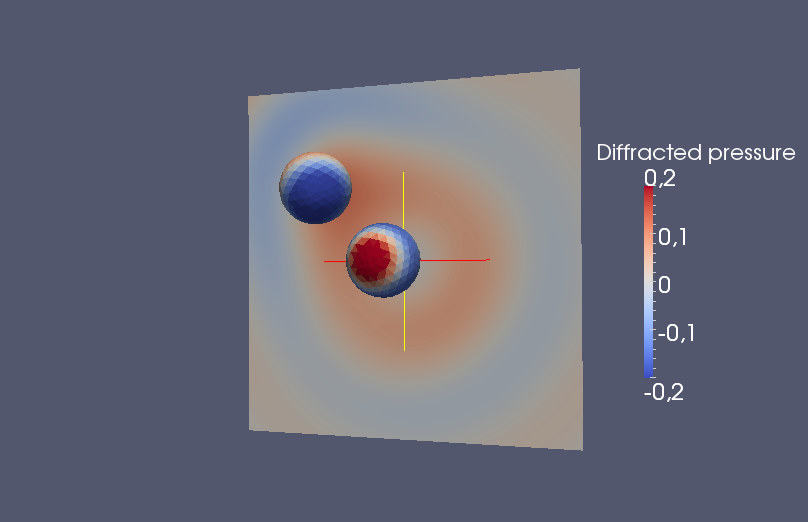}
	\caption{Real part of the pressure field for the BEM solution (left) and the RB solution (right), with a
basis of size 10. The difference between the two fields is less than $10^{-15}$ in infinity norm.}
\label{fig:comp2sph}
\end{figure}

\subsection{Error bound curves}

We present the error bound curves for test case
(i) with a general-purpose RB, $\#\mathcal{P}_{\textnormal{trial}}=100$, $(\hat{N},\hat{\sigma},\sigma)=(2,7,49),(3,10,100),(4,20,169)$,
and $(5,30,256)$ in Figure \ref{fig:fig1} and for test case (ii) with a goal-oriented RB,
$\#\mathcal{P}_{\textnormal{trial}}=225$, $\hat{N}=8$, $\hat{\sigma}=60$, and $\sigma=1681$
in Figure \ref{fig:fig2}.

\begin{figure}[h!]
	\centering
	\includegraphics [width=7.5cm] {err_est2.pdf}
	\includegraphics [width=7.5cm] {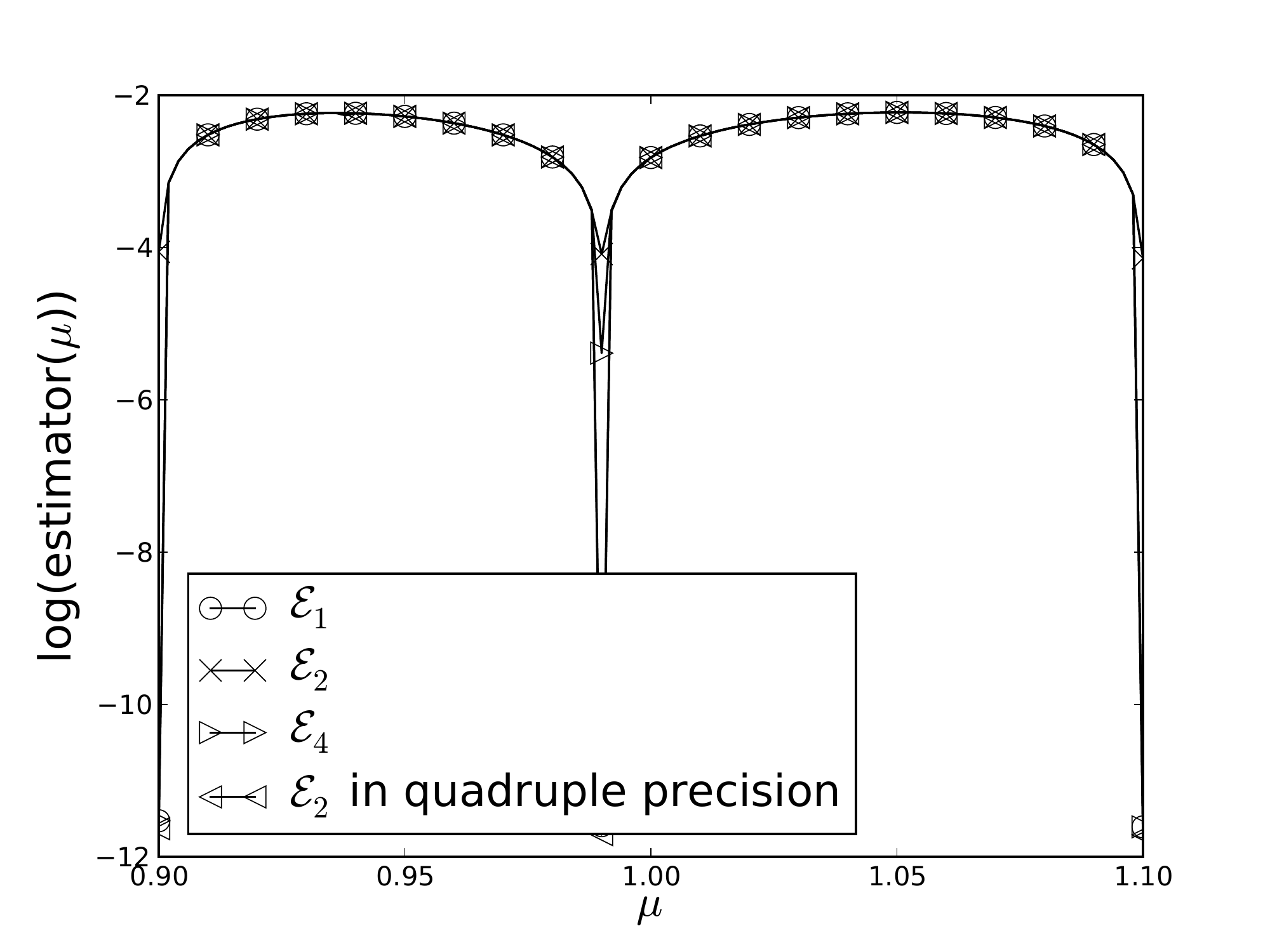}\\
	\includegraphics [width=7.5cm] {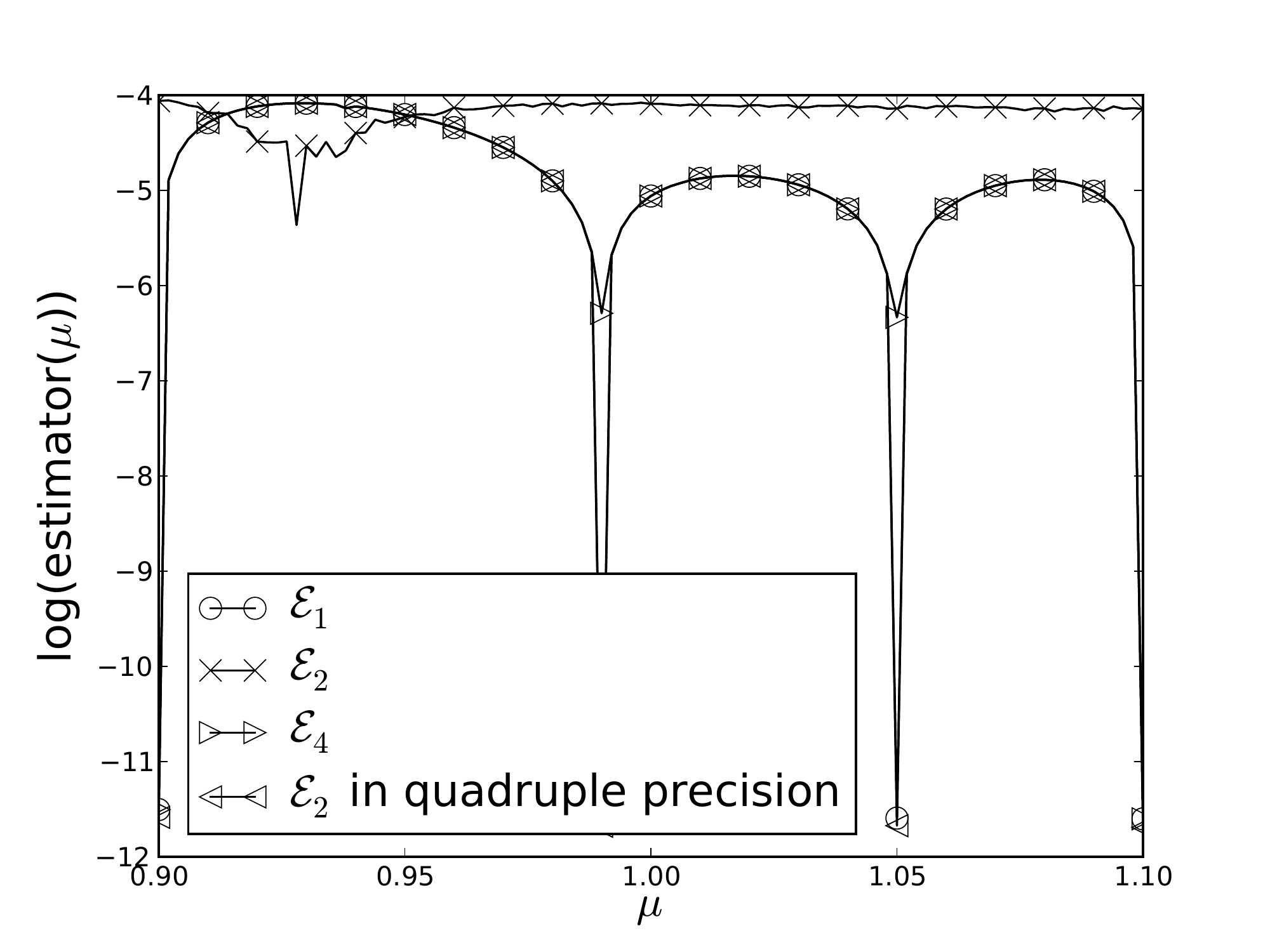}
	\includegraphics [width=7.5cm] {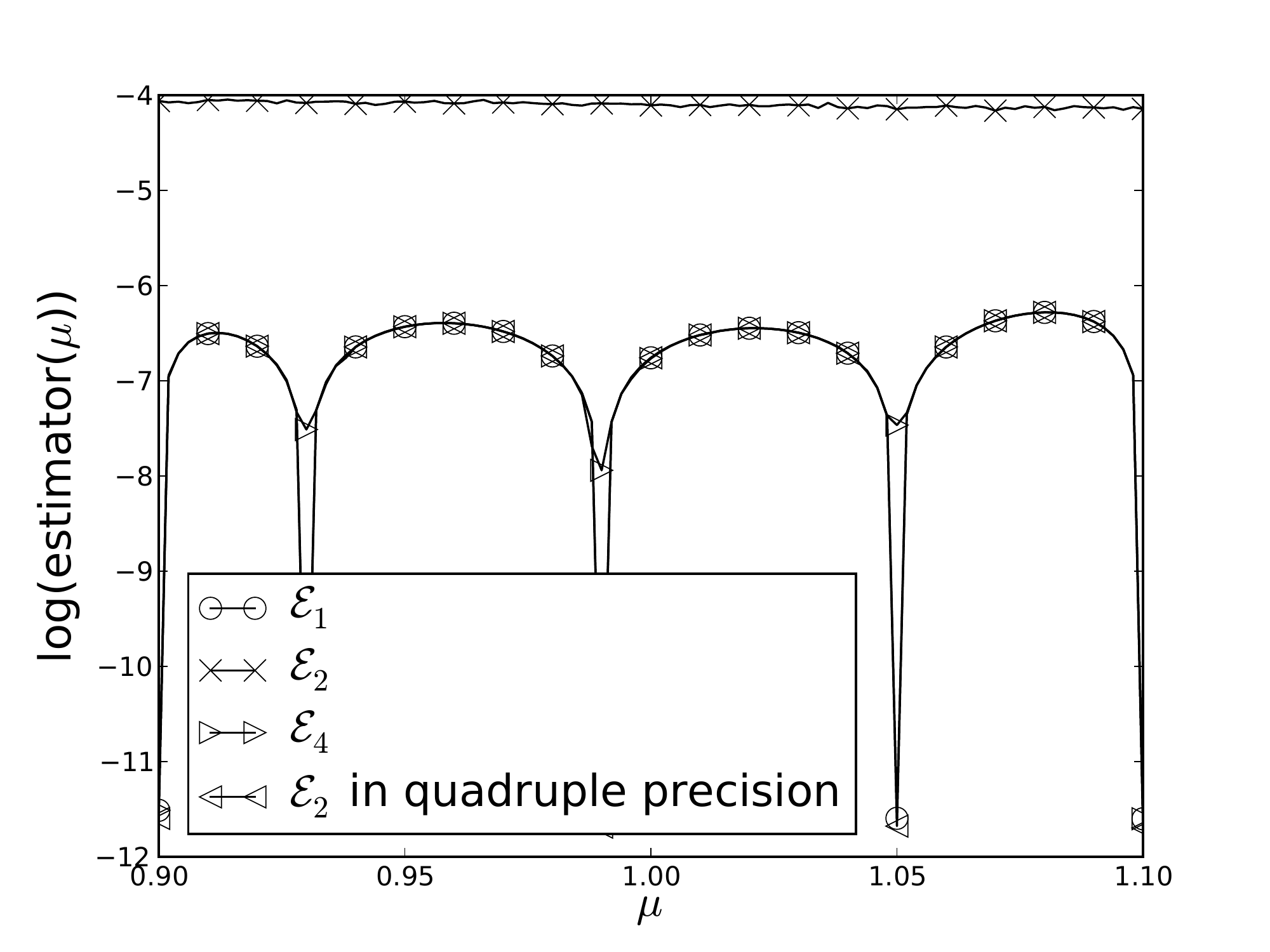}
	\caption{Error bound curves with respect to the impedance coefficient, with $\hat{N}$ equal to
$2$, $3$, $4$, and $5$ (from left to right and top to bottom). The curve for $\mathcal{E}_2$ computed in quadruple precision superimposes to $\mathcal{E}_1$.}
\label{fig:fig1}
\end{figure}

\begin{figure}[h!]
	\centering
	a)\includegraphics [width=7cm] {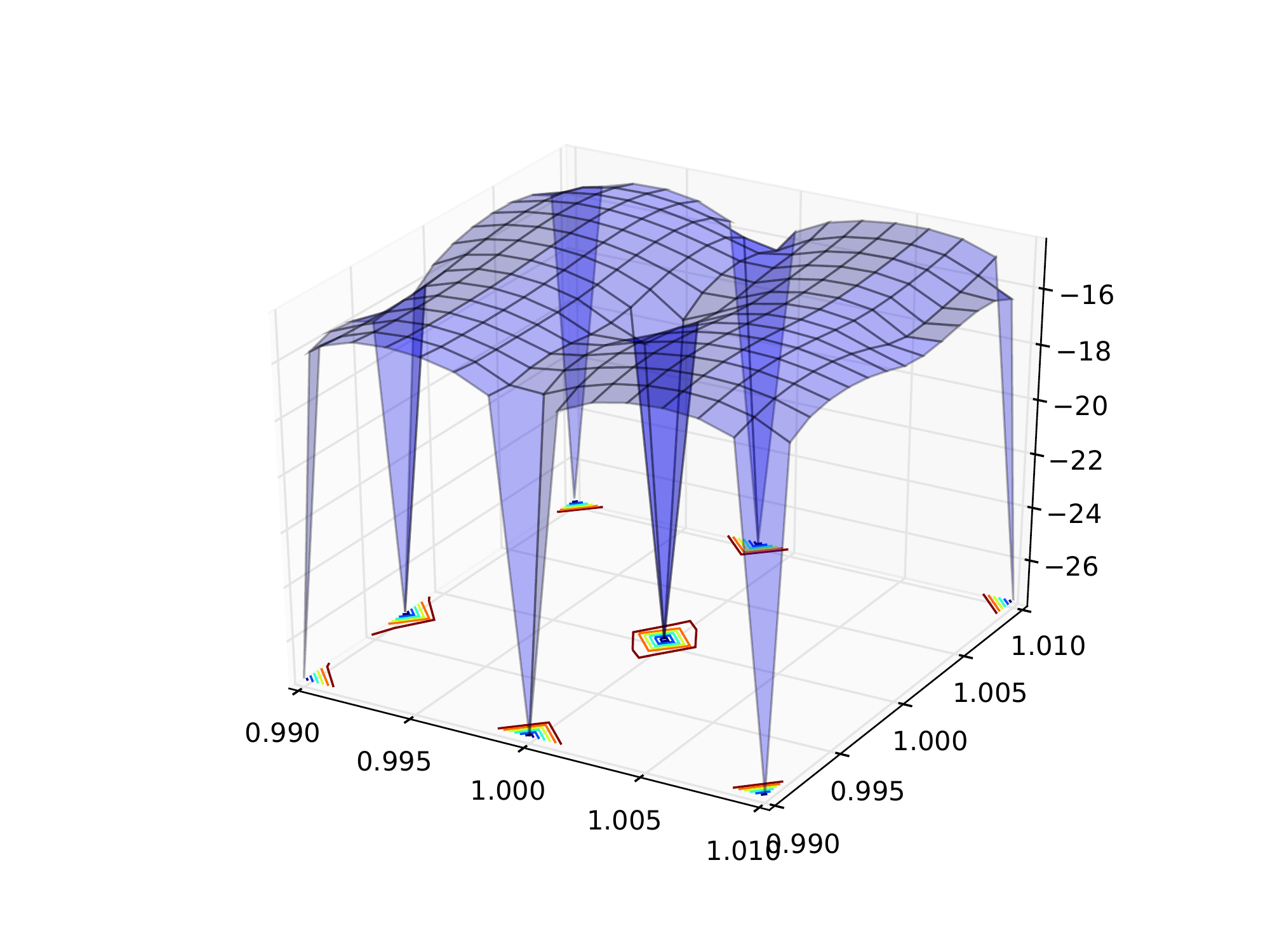}
	b)\includegraphics [width=7cm] {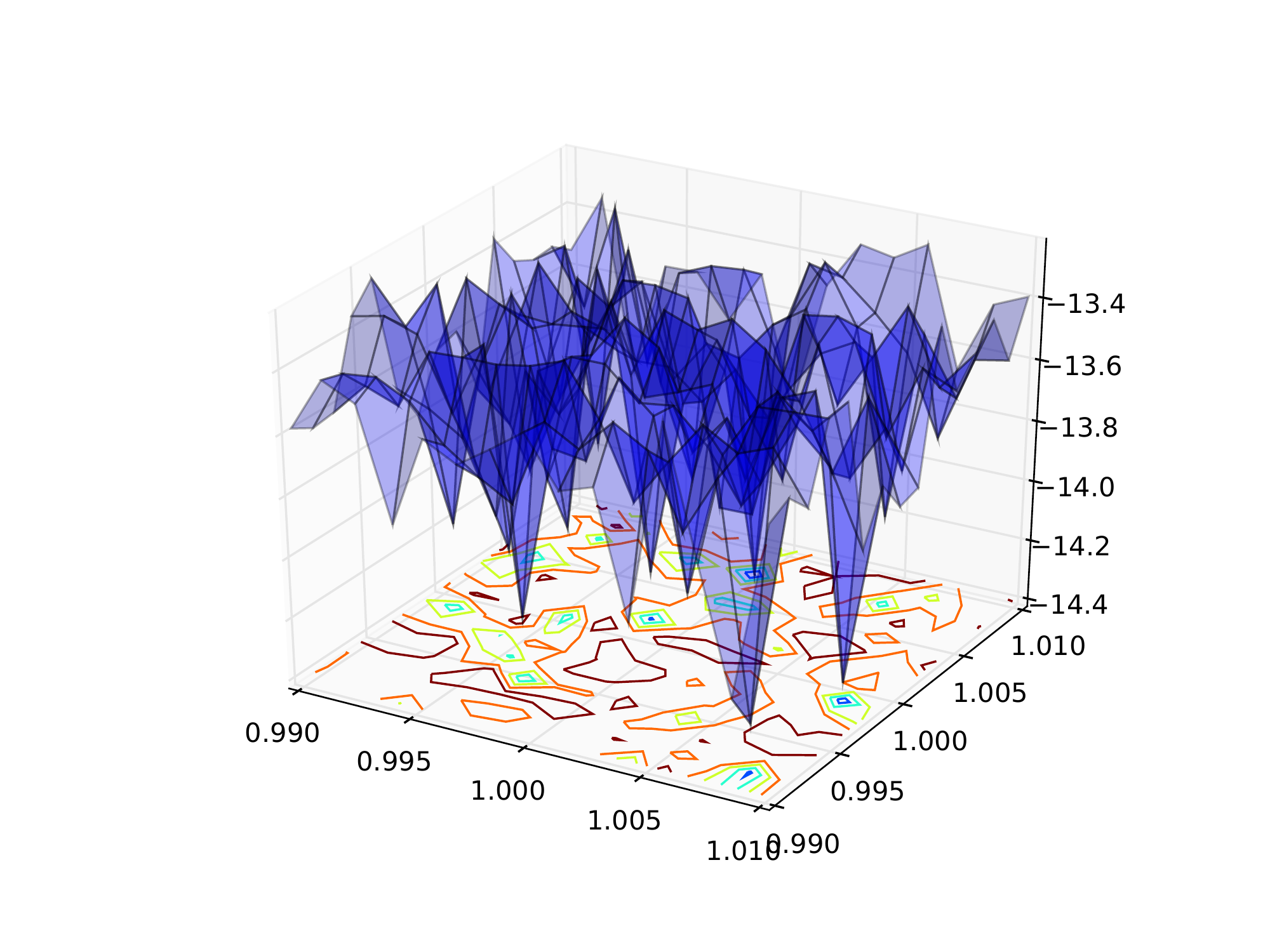}
	c)\includegraphics [width=7cm] {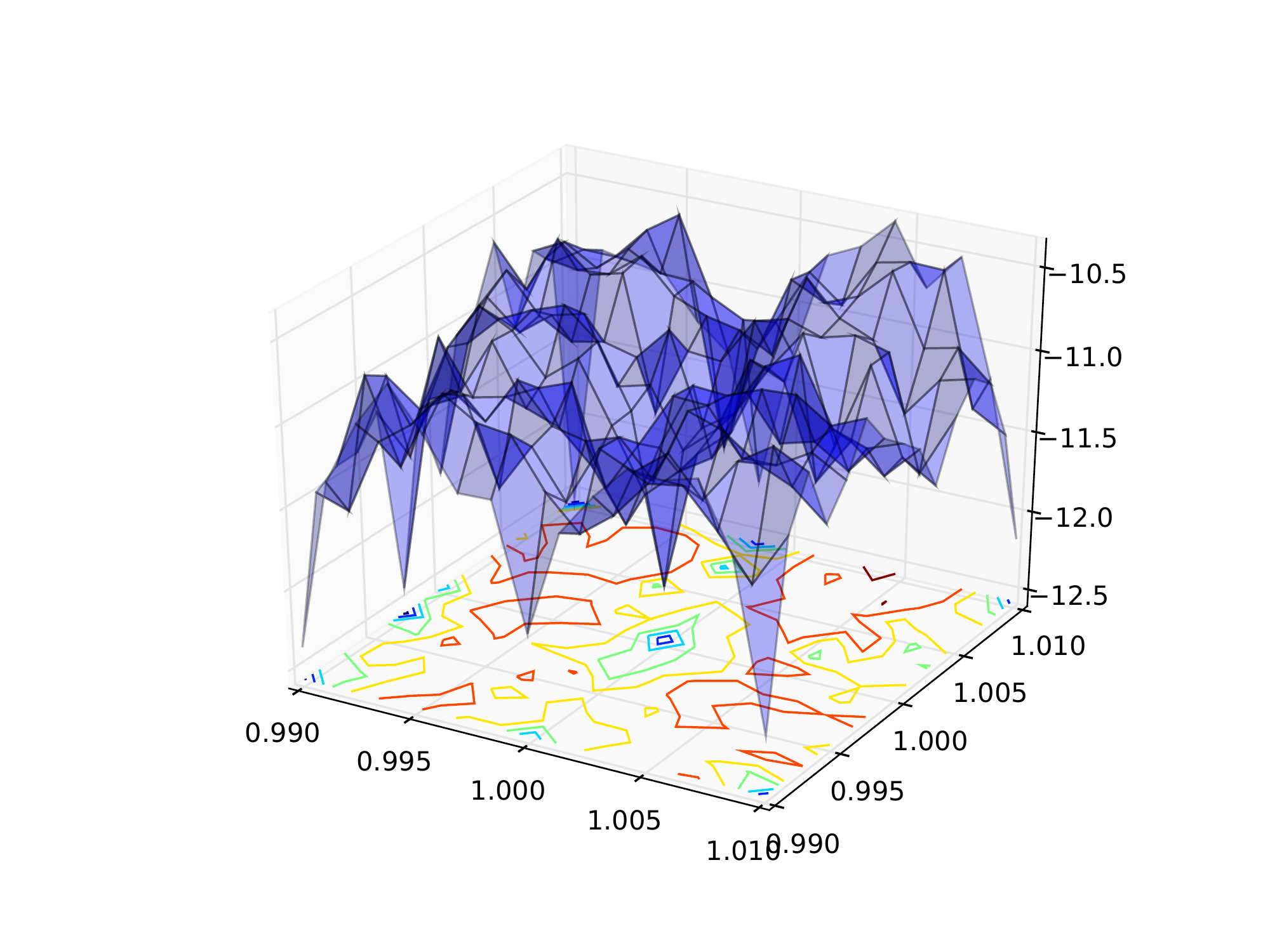}
	d)\includegraphics [width=7cm] {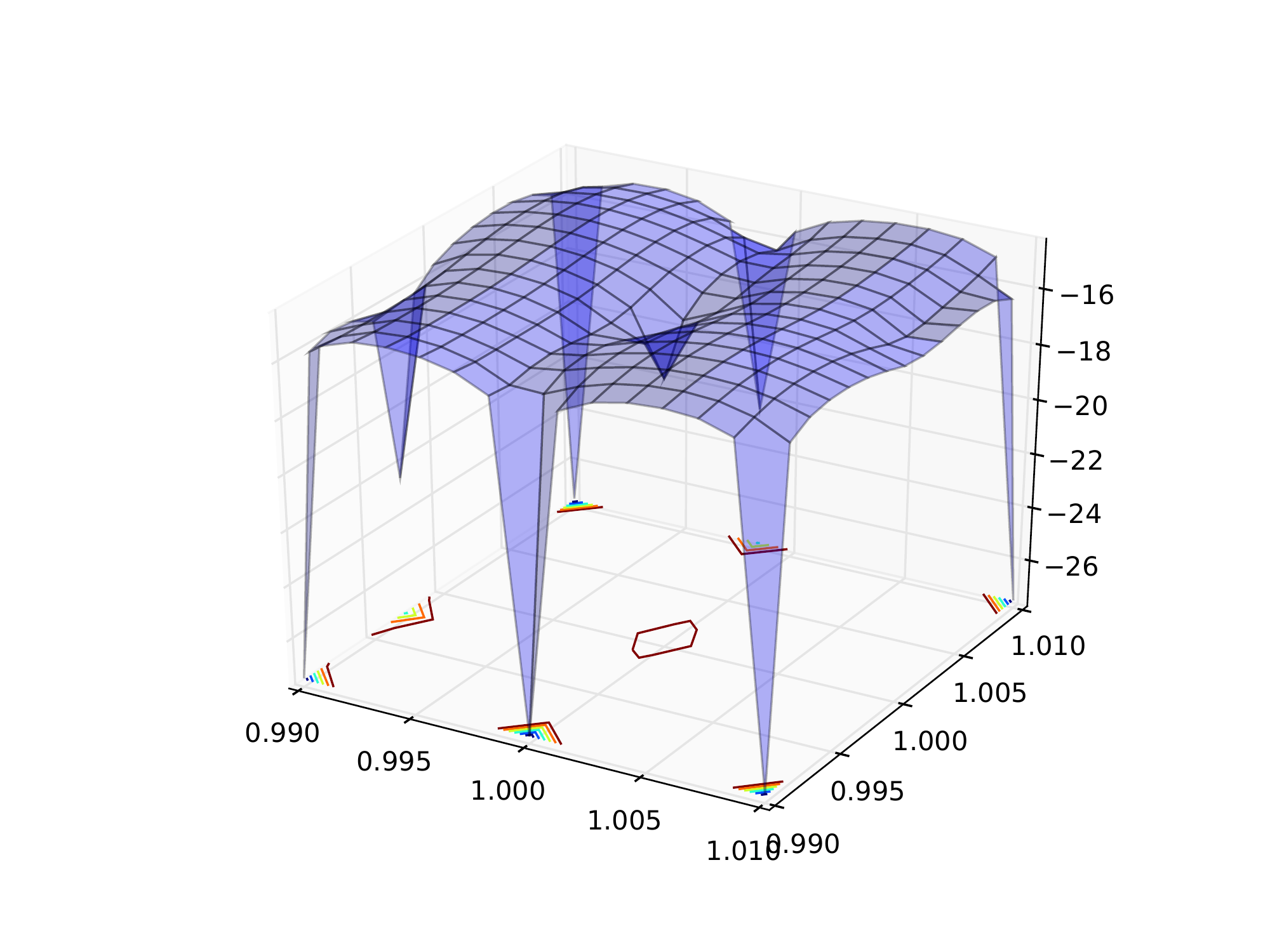}
	e)\includegraphics [width=7cm] {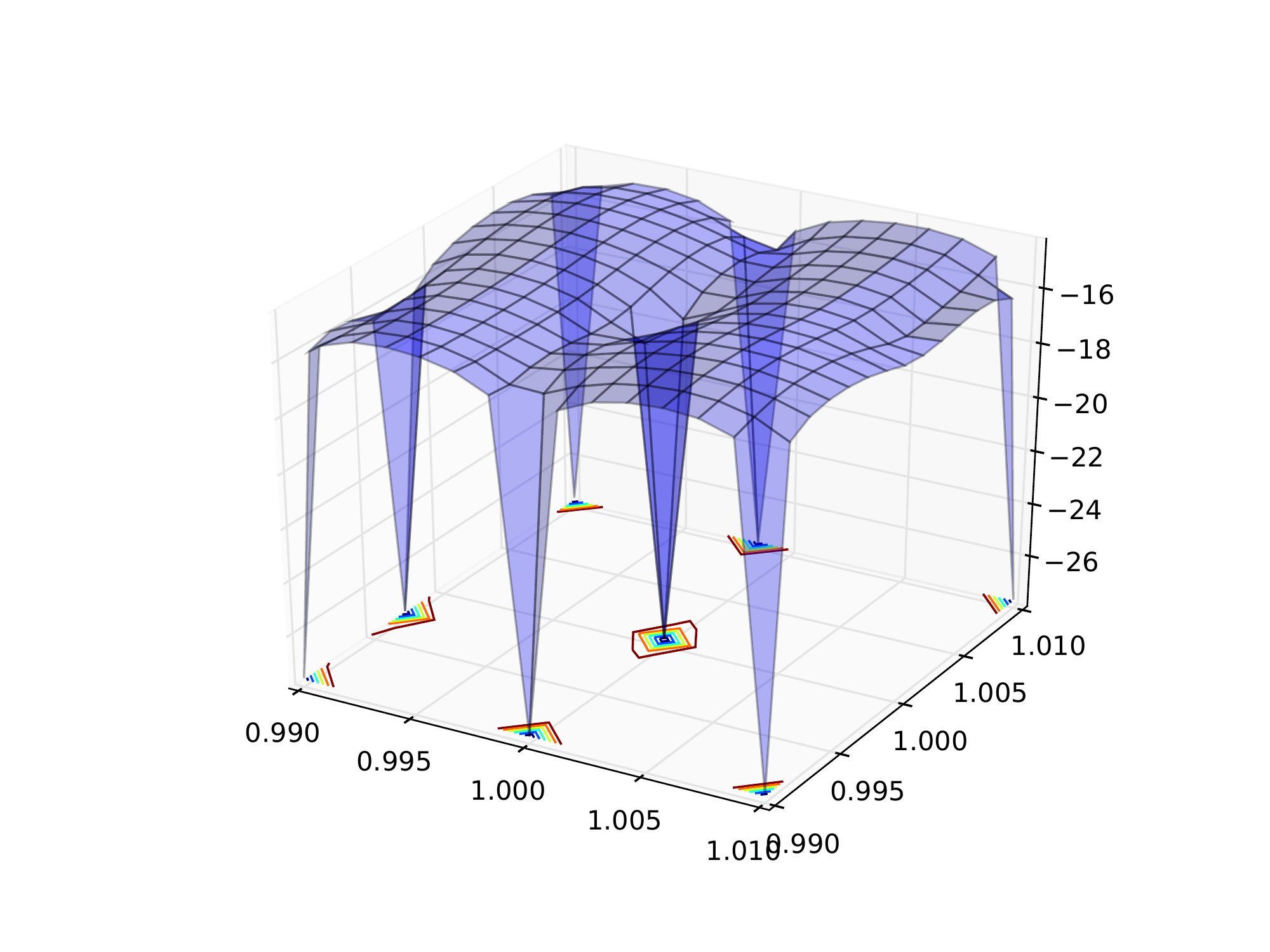}
	\caption{Error bound curves (logarithmic scale) as a function of the impedance coefficients: a) $\mathcal{E}_1$, b)
 $\mathcal{E}_2$, c) $\mathcal{E}_3$, d) $\mathcal{E}_4$, and e) $\mathcal{E}_2$ computed in quadruple precision.}
\label{fig:fig2}
\end{figure}

In test case (i), the classical formula $\mathcal{E}_2$ exhibits quite poor performances, since it cannot
compute values below $10^{-4}$. This is explained by the values of the inf-sup constant which are around $10^{-6}$.
Furthermore, in agreement with Remark \ref{direct_case}, the lowest computable values of $\mathcal{E}_1$ and $\mathcal{E}_2$
differ by $8$ orders of magnitude.
In test case (ii), the behavior of formula $\mathcal{E}_3$ is quite poor, and we do not observe the level of accuracy we observed
so far for $\mathcal{E}_3$. Here, the matrix $T$ defined in \eqref{eq:sysT} is so ill-conditioned that the numerical errors introduced
by its resolution are larger than the ones introduced by the formula $\mathcal{E}_2$.
Furthermore, the formula $\mathcal{E}_4$ exhibits, as before, a very good performance.
We see in Figure \ref{fig:fig2} that $\underset{\mu\in\mathcal{P}_{\rm select}}{\textnormal{argmax}}\left(\mathcal{E}_4(\mu)\right)=(1,1)$
and $\mathcal{E}_4(1,1)\approx 10^{-16}$; therefore, the formula $\mathcal{E}_4$ with
$\hat{\sigma}=60$ is valid for computing the error bound in Algorithm 1 with $\rm tol=10^{-16}$.

The behavior of $\mathcal{E}_4$ when $\hat{\sigma}$
increases is investigated in Figure \ref{fig:fig1bis} for test case (i). We consider the values $\hat{\sigma}=14,30,40$ and $50$.
These four values lead to the same local maxima, and increasing $\hat{\sigma}$ allows the formula $\mathcal{E}_4$ to be valid for smaller
tolerances (respectively $5\times 10^{-8}$, $10^{-8}$, $8\times 10^{-9}$ and $2\times 10^{-9}$).
Another interesting observation comes from considering the fourth plot in Figure \ref{fig:fig1} and the first plot in Figure
\ref{fig:fig1bis}:
the classical formula $\mathcal{E}_2$ requires $16$ offline resolutions of \eqref{eq:compute_G} and stagnates at $10^{-4}$
while the formula $\mathcal{E}_4$ with $\hat\sigma = 14$ only requires $14$ offline resolutions of \eqref{eq:compute_G}
and is valid for tolerances down to $5\times 10^{-8}$. This shows that at least in some regimes, the new formula $\mathcal{E}_4$ is
valid for lower tolerances than the classical formula $\mathcal{E}_2$, and requires less precomputations.
However, contrary to $\mathcal{E}_2$, using $\mathcal{E}_4$ requires that all the quantities $V_r$ defined in $\eqref{eq:V_r}$ be
recomputed when adding a new vector to the reduced basis.

\begin{figure}[h!]
	\centering
	\includegraphics [width=7.5cm] {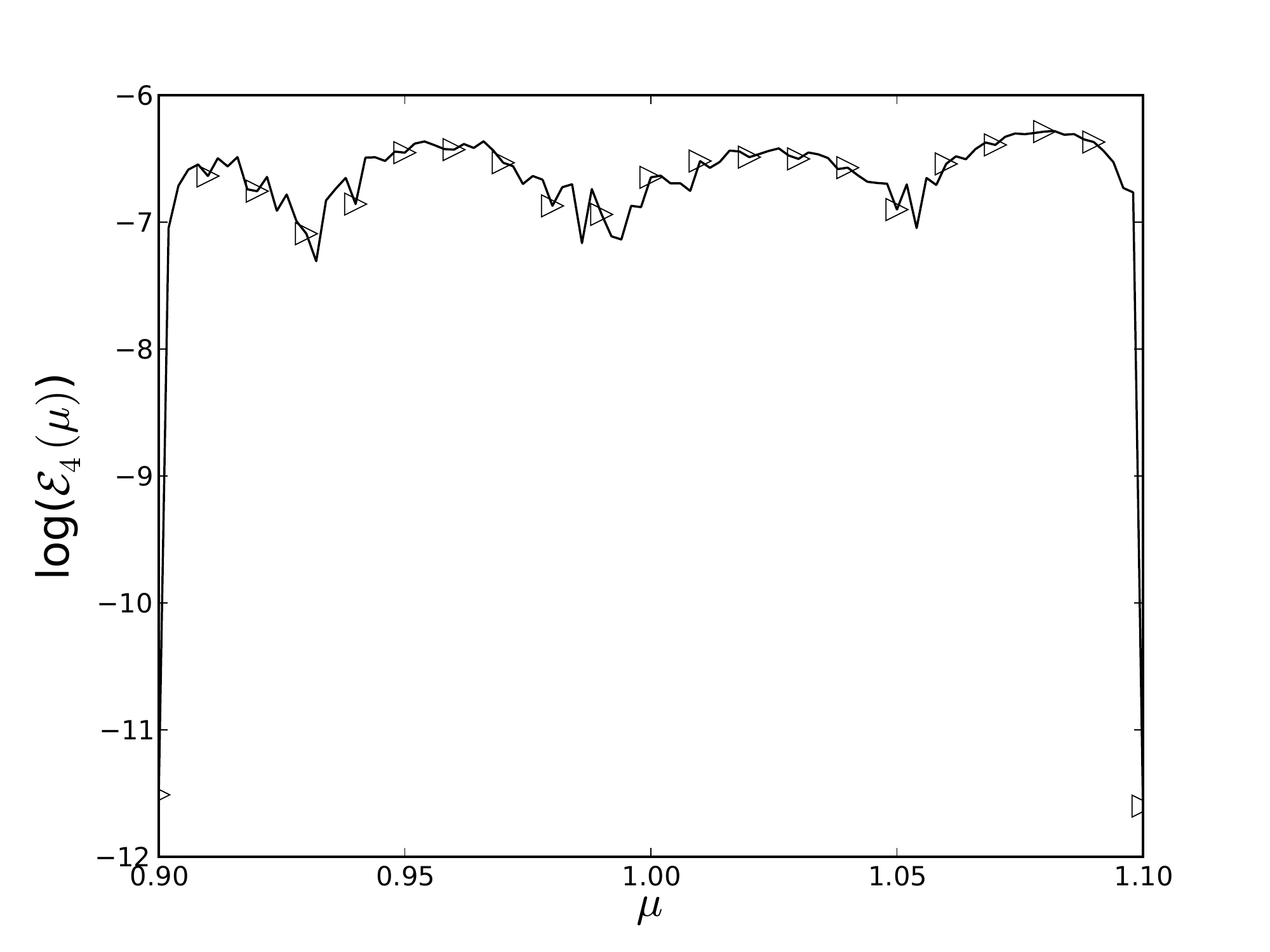}
	\includegraphics [width=7.5cm] {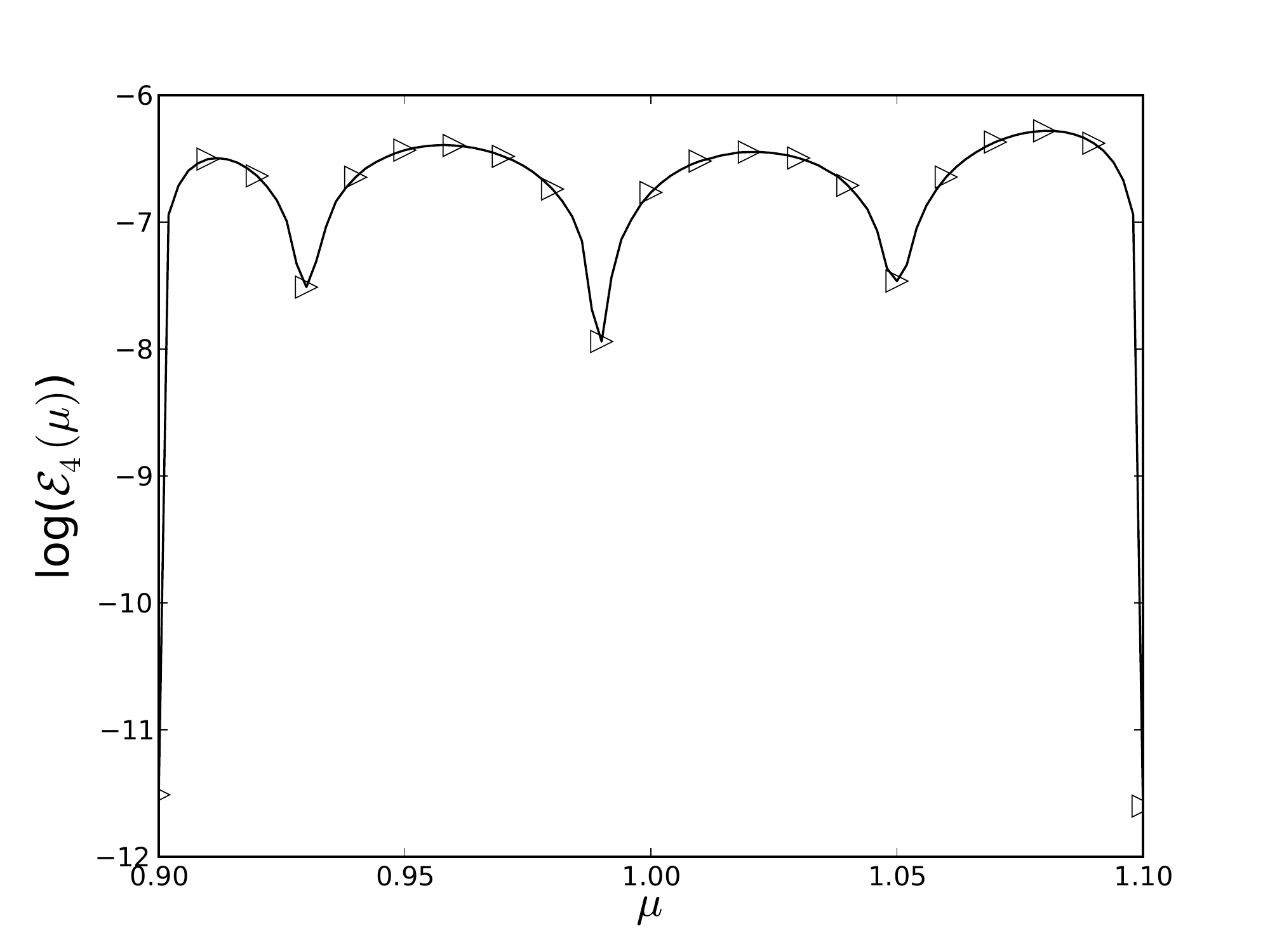}
	\includegraphics [width=7.5cm] {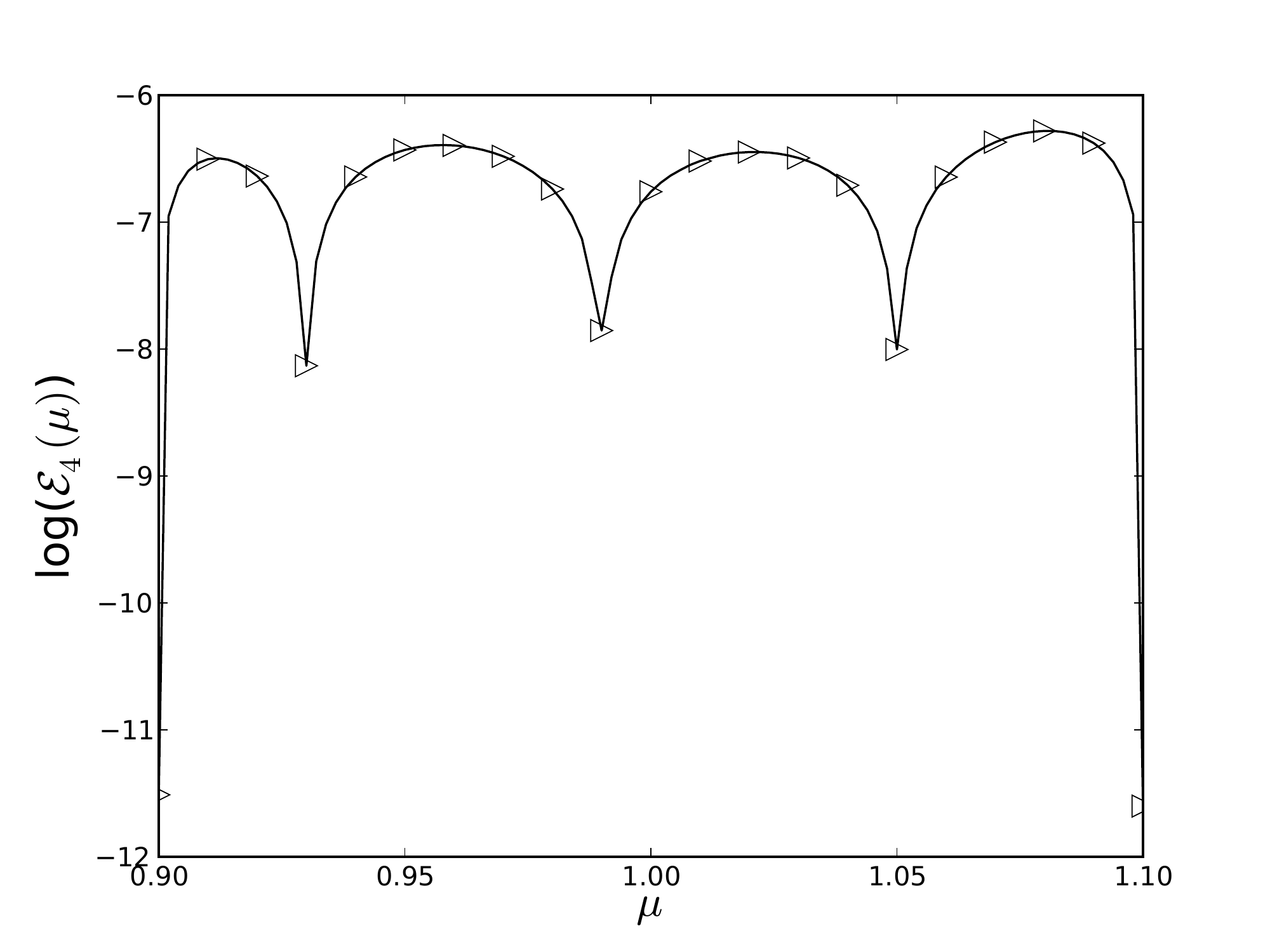}
	\includegraphics [width=7.5cm] {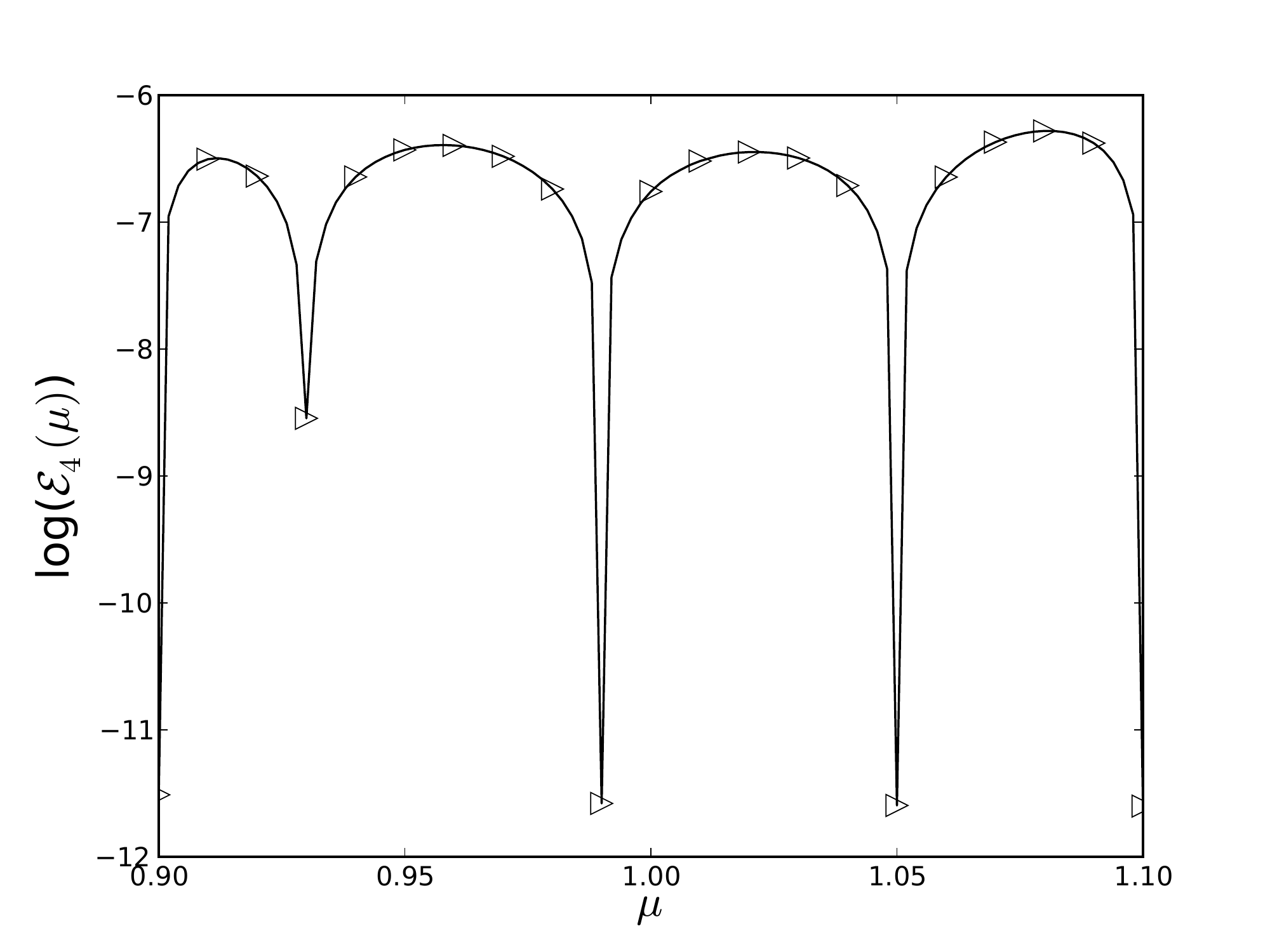}
	\caption{Error bound curve for $\mathcal{E}_4$ with respect to the impedance coefficient, with $\hat{N}=5$ and
$\hat{\sigma}$ equal to $14$, $30$, $40$, and $50$ (from left to right and top to bottom).}
\label{fig:fig1bis}
\end{figure}

\section*{Conclusion}
In this work, we have extended the ideas of \cite{casenave} by proposing a more stable numerical procedure, using
the empirical interpolation method, to represent the a posteriori error bound in the reduced basis method
as a linear combination of its values at given parameter values, called interpolation points.
Moreover, the proposed method provides a way of choosing the interpolation points, and yields better accuracy levels than the
classical a posteriori error bound and than the procedure proposed in \cite{casenave}.
Besides, our new procedure may require less precomputations than the classical a posteriori error bound.
The new error bound derived herein can be of particular interest in two situations:
(i) when the stability constant of the original problem is very small (this is the case in many practical problems),
(ii) when very accurate solutions are needed,
(iii) when considering a nonlinear problem (for which, in some cases, no error bound is possible until a very tight tolerance is reached, see~\cite{yano}).

\section*{Acknowledgement}
This work was supported by EADS IW. The authors wish to thank Anthony Patera for fruitful discussions.

\end{document}